\documentclass[12pt, oneside]{article}
\usepackage[T1]{fontenc}
\usepackage[utf8]{inputenc}
\usepackage{amstext,amsmath, amssymb, amsthm, mathrsfs, mathtools}
\usepackage[a4paper, total={6.2in, 9.1in}, left={1.3in}, right={1.3in}, top={1.1in},bottom={1.3in}, footskip={.5in}]{geometry}
\usepackage{titlesec}
\usepackage{authblk}
\usepackage{enumerate}
\usepackage[dvipsnames]{xcolor}
\usepackage{comment}
\usepackage[colorlinks=true,
    linkcolor=black,
    citecolor=RoyalBlue,
    urlcolor=cyan]{hyperref}
\usepackage[capitalise]{cleveref}
\usepackage{etoolbox}
\usepackage{stmaryrd}

\usepackage{setspace}
\allowdisplaybreaks

\makeatletter
\newcommand\@authorinfo{}
\renewcommand{\author}[2][]{%
  \g@addto@macro\@author{#2\and}%
  \g@addto@macro\@authorinfo{%
    \par\bigskip
    \noindent{\scshape #2}\par
  }%
}
\newcommand{\address}[1]{%
  \g@addto@macro\@authorinfo{#1\par}%
}
\newcommand{\email}[1]{%
  \g@addto@macro\@authorinfo{\texttt{#1}\par}%
}
\AtEndDocument{%
  \par\bigskip
  \@authorinfo
}
\makeatother

\makeatletter
\renewcommand{\maketitle}{%
    \vspace*{1em}
    \begin{center}
        {\normalfont\bfseries\MakeUppercase{\@title}\par}%
        \vskip 1em
        {\normalfont\small \@date \par}%
    \end{center}
}
\makeatother

\renewenvironment{abstract}{%
    \noindent
    \begin{list}{}{%
        \setlength{\leftmargin}{2.5em}%
        \setlength{\rightmargin}{2.5em}%
        }
        \item\relax
        \textsc{Abstract. }\ignorespaces
        }{%
  \end{list}
}

\titleformat{\section}
  {\normalfont\scshape\centering}
  {\thesection.}{0.5em}{}

\titleformat{\subsection}[runin]
  {\normalfont\bfseries}
  {\thesubsection.}{0.5em}{}[.]

\usepackage[backend=bibtex]{biblatex}
\newtheorem{theorem}{Theorem}[section]
\newtheorem{lemma}[theorem]{Lemma}

\newtheorem{proposition}[theorem]{Proposition}
\newtheorem{corollary}[theorem]{Corollary}
\newtheorem{definition}[theorem]{Definition}

\theoremstyle{definition} 
\newtheorem{remark}[theorem]{Remark}

\numberwithin{equation}{section}

\DeclareMathOperator{\nR}{\mathbb{R}}
\DeclareMathOperator{\nC}{\mathbb{C}}
\DeclareMathOperator{\nN}{\mathbb{N}}
\DeclareMathOperator{\nZ}{\mathbb{Z}}
\DeclareMathOperator{\nS}{\mathbb{S}}
\DeclareMathOperator{\nH}{\mathbb{H}}

\DeclareMathOperator{\nU}{\mathbb{U}}

\DeclareMathOperator{\Aut}{Aut}
\DeclareMathOperator{\End}{End}
\DeclareMathOperator{\tr}{tr}

\DeclareMathOperator{\GL}{GL}
\DeclareMathOperator{\BO}{O}
\DeclareMathOperator{\U}{U}

\DeclareMathOperator{\SU}{SU}

\DeclareMathOperator{\Rm}{Rm}
\DeclareMathOperator{\Ric}{Ric}
\DeclareMathOperator{\R}{R}

\DeclareMathOperator{\Div}{div}

\title{On the singularity formation of gauge fields coupled to Ricci flow}

\date{}

\begin{document}

\maketitle

\begin{center}
    {\footnotesize \MakeUppercase{Andoni Royo Abrego}}
\end{center}

\begin{abstract}
    We study a family of coupled geometric evolution equations describing the deformation of a Riemannian metric and a non-abelian gauge field on closed manifolds, which generalizes the Ricci--Yang--Mills flow. We derive interior curvature estimates, find a preserved integral curvature condition, and discover a scaling invariant monotone functional analogue to Perelman's entropy. In particular, we prove strong dominance of the Riemannian curvature over the gauge curvature at finite-time singularities in all dimensions. We also provide a non-trivial explicit example of a shrinking self-similar solution on a $\SU(2)$ bundle over $S^4$. 
\end{abstract}

\vspace{0.3cm}

\tableofcontents

\hfill

\pagestyle{plain}

\section{Introduction and main results}
Let $E \to M$ be a vector bundle over a closed manifold of dimension $n \geq 2$ with a fixed fiber-wise inner product $\langle\cdot,\cdot\rangle$ on $E$. We study the deformation of Riemannian metrics $g$ on $M$ and connections $A$ on $E$ along the coupled system
\begin{equation}
    \label{eq: * kappa}
    \tag{$*_\kappa$}
    \begin{split}
        \frac{\partial g}{\partial t} &= -2\Ric + \kappa \, Z
        \\
        \frac{\partial A}{\partial t} &= - D^*F
    \end{split} \qquad,
\end{equation}
where $\Ric$ denotes the Ricci curvature of $g$, $F$ and $D^*$ are the curvature and the adjoint exterior covariant derivative of $A$ respectively, $\kappa$ is a positive coupling parameter depending smoothly on time, and $Z$ is the symmetric tensor
\begin{equation*}
    Z_{ij} \coloneqq g^{pq}\langle F_{ip}, F_{jq}\rangle = g^{pq}F^\alpha_{ip\beta}F^\beta_{jq\alpha} \,.
\end{equation*}
Equation \eqref{eq: * kappa} is arguably the most natural coupling between Ricci flow and Yang--Mills flow, and it shows up in the context of renormalization group flows, see \cite{CarforaGuenther,RenormalizationGF1,RenormalizationGF2} and references therein. Since $Z$ is positive semi-definite, it acts as a force term against positive Riemannian curvature; one may thus hope that coupling a gauge field slows down singularity formation in Ricci flow, or even prevent it. Indeed, we will see that this may happen in the presence of nontrivial topology of the bundle $E$.

The geometric flow \eqref{eq: * kappa} was first studied by Streets \cite{StreetsThesis} and Young \cite{YoungStability} in the case of constant coupling parameter $\kappa(t) \equiv 1$, where the equation arises from modifying Ricci flow of a Kaluza--Klein metric on a principal bundle to keep fibers fixed. Due to the scaling of the equation, it seems somewhat more natural to use a coupling parameter scaling like
\begin{equation*}
    [\,\kappa\,] = [\,time\,] = [\,distance\,]^{2} = [\,volume\,]^{2/n} = [\,curvature\,]^{-1} \,.
\end{equation*}
Natural choices are, for example, an affine function $\kappa(t) = t_0 - t$ for some $t_0 > 0$, or $\kappa(t) = |M_t|^{2/n}$. One could also choose $\kappa(t)$ to keep the volume fixed along the flow. For the most part of the paper, we will work with a general coupling parameter and only assume bounds
\begin{equation}
    \label{eq: kappa condition}
    0 < \kappa_0 \leq \kappa(t) \leq \kappa_1 < \infty
\end{equation}
for all times $t \geq 0$.

The case where $M$ is two dimensional and $E$ has compact abelian structure group (i.e. tori bundles) was thoroughly studied by Streets \cite{StreetsRYMFSurfaces1,StreetsRYMFSurfaces2}. In this case, the flow is equivalent to \emph{generalized Ricci flow} and it can be reduced to a coupled evolution for a conformal factor and a smooth function; see \cite{GarciaStreetsBook,StreetsMonotonicity,LiGRFestimates,LiGRFheatkernel} for further references on generalized Ricci flow. Using the techniques of Struwe \cite{StruweSurfacesFlow}, he was able to give a complete picture of the long-time behavior of the flow. In parallel, Young \cite{YoungStability,YoungNilpotent} studied stability properties near Einstein Yang--Mills metrics, stationary solutions, on $\U(1)$-bundles over closed surfaces.

In this work, we are interested in manifolds of arbitrary dimensions and general structure groups. While in the abelian case the Yang--Mills flow is linear and equivalent to the heat equation for 2-forms, we aim to understand the non-linear interaction between the Riemannian curvature and the gauge curvature. One of the main difficulties in the analysis of the coupled flow \eqref{eq: * kappa} is that many of the fundamental tools in the separate study of Ricci flow and Yang--Mills flow are not available. Most notably, the Yang--Mills energy 
\begin{equation*}
    \mathcal{YM}_{g}(A) = \int_{M} |F|^2 \, d\mu
\end{equation*}
is no longer monotone and point-wise positivity conditions on the Riemannian curvature, such as positive scalar curvature, are not preserved. 

Both the evolution of the metric and the connection in \eqref{eq: * kappa} differ from Ricci flow and Yang--Mils flow respectively in lower orders terms, and short-time existence and uniqueness for smooth initial data is guaranteed by a simultaneous gauge fixing trick, see \cite[Section 6.1]{StreetsThesis} for example.
It was proven by Hamilton \cite{Hamilton82} that Ricci flow can in general develop a singularity at some finite time $T > 0$, which is characterized by curvature blow-up
\begin{equation*}
    \limsup_{t \to T}\sup_M |\Rm(t)| = \infty \,.
\end{equation*}
Indeed, there exists a positive constant $C$ such that
\begin{equation*}
    \sup_M|\Rm(t)| \geq \frac{C}{T-t} \,.
\end{equation*}
In contrast, Yang--Mills flow admits global solutions and converges to Yang--Mills connections as $t \to \infty$ in subcritical dimensions $n < 4$, see Rade \cite{RadeYMF} or Daskalopoulos \cite{DaskalopoulosYMflow}. On the other hand, examples that blow-up in finite time were constructed by Naito \cite{NaitoYMFblowup} in supercritical dimensions $n > 5$. In this case, parabolic estimates show that the gauge curvature blows up
\begin{equation*}
    \limsup_{t \to T}\sup_M |F(t)| = \infty
\end{equation*}
at the maximal time of existence. Waldron \cite{WaldronLongTime} was able to rule out energy concentration in critical dimension $n = 4$ and, building up on work of Struwe \cite{Struwe4DYMflow}, proved that finite-time singularities of Yang--Mills flow do not occur. This is in stark contrast with harmonic map heat flow in critical dimension $n=2$. 

It is therefore sensible to expect that for $n \leq 4$, 
Ricci flow dominates over Yang--Mills flow at finite-time singularities of \eqref{eq: * kappa}. Our first result confirms such suspicion in all dimensions. 
\begin{theorem}
    \label{Theorem A}
    Let $\bigl(g(t), A(t)\bigr)$ be a smooth solution of \eqref{eq: * kappa} with a coupling parameter satisfying \eqref{eq: kappa condition} on a maximal time interval $t \in [0, T)$. Then, either $T = \infty$, or else there exists a positive constant $C$ such that
    \begin{equation*}
        \sup_M|\Rm(t)| \geq \frac{C}{T-t} \,.
    \end{equation*}
    Moreover, for any $\alpha \geq 1$, a blow-up rate
    \begin{equation*}
        \sup_M|\Rm(t)| = \BO\bigl((T-t)^{-\alpha}\bigr) \qquad \text{as} \quad t \to T
    \end{equation*}
    of the Riemannian curvature implies that
    \begin{equation*}
        \sup_M|F(t)| = \BO\bigl((T-t)^{-\frac{\alpha} {2}}\bigr) \qquad \text{as} \quad t \to T \,.
    \end{equation*}
\end{theorem}
In particular, we have that the gauge curvature $|F|$ and all its derivatives remain bounded as long as $|\Rm|$ is under control. If $|\Rm|$ blows up in finite time, then $|F|$ can only blow up, if any, at a much smaller rate. Furthermore, the Riemannian curvature blows up at least at the scaling invariant rate suggested by the Ricci flow. This demonstrates a strong dominance of Ricci flow over the gauge field at finite-time singularities.

Concerning short-time behavior, we establish a lower bound on the time of existence purely in terms of the initial data, together with the corresponding parabolic estimates.
\begin{theorem}
    \label{Theorem Short Time}
    Let $(g, A)$ be arbitrary smooth initial data, $\kappa(t)$ a smooth positive coupling parameter and let
    \begin{equation*}
        K \coloneq \|\Rm\|_{L^\infty(M)} + \|F\|_{L^\infty(M)} + \kappa(0)\|F\|_{L^{\infty}(M)}^2 + \kappa^{1/2}(0)\|\nabla F\|_{L^{\infty}(M)}^2 \,.
    \end{equation*}
    Then, there exist positive constants $\tau=\tau(n)$, $c = c(n,\kappa(0))$ and $C(n,m,\kappa(0))$ such that a smooth solution to \eqref{eq: * kappa} exists on a time interval $[0, \tau K^{-1}]$ and the estimate
    \begin{equation*}
        \|\nabla^{m-1}\Rm(t)\|_{L^\infty(M)}^2 + \kappa(0)\|\nabla^{m}F(t)\|_{L^\infty(M)}^2 \leq \frac{CK}{t^{m}}
    \end{equation*}
    holds for all times $t \in \bigl(0, cK^{-1}\bigr]$ and positive integer $m$.
\end{theorem}
We remark that the dependence of $K$ on $|\nabla F|$ is not surprising: while $\Rm$ is a second order operator on $g$, the gauge curvature $F$ is only first order on $A$.

As mentioned above, the Yang--Mills energy is not monotone under \eqref{eq: * kappa}, and scalar curvature lower bounds are not preserved. If the coupling parameter $\kappa(t)$ is non-increasing, we obtain the following theorem, which can be regarded as a preserved integral curvature condition.
\begin{theorem}
    \label{Theorem preserved integral inequality}
    Let $\bigl(g(t), A(t)\bigr)$ be a solution of \eqref{eq: * kappa} satisfying \eqref{eq: kappa condition} and $\kappa'\leq 0$ for all $t \in [0,T)$. Then, the inequality
    \begin{equation}
        \label{eq: preserved integral inequality}
        \frac{\kappa}{4} \int_M |F|^2 \, d\mu \leq \int_M \R \, d\mu
    \end{equation}
    is preserved in time.
\end{theorem}
This indicates an interaction between the topology of the bundle $E$ and the behavior of the coupled flow. For instance, let $(M,g)$ be a compact oriented Riemannian four-manifold of positive scalar curvature and fix an initial connection $A$ and constant $\kappa > 0$ sufficiently small so that \eqref{eq: preserved integral inequality} holds true. Then, if we run the flow \eqref{eq: * kappa} with $\kappa(t) = \kappa$, it follows from \cref{Theorem preserved integral inequality} and Chern--Weil theory that
\begin{equation*}
    \int_{M} \R \, d\mu \geq 2\pi^2 \kappa |\tau(E)|
\end{equation*}
holds for all positive times, where $\tau(E) \in \nZ$ is the instanton number of $E$. In particular, if $E$ is non-trivial, the total scalar curvature has a uniform positive lower bound and the manifold $(M,g(t))$ can not shrink down to a round point. An explicit example will be discussed in \cref{An important example}.

One of the main obstacles in Street's study of \eqref{eq: * kappa} with $\kappa \equiv 1$, is that despite having a monotone coupled energy, the natural analogue of Perelman's entropy is not monotone, see \cite[Chapter 3]{StreetsThesis}. Such entropy monotonicity is a fundamental tool in Perelman's work on singularity formation along Ricci flow \cite{Perelman1,Perelman2,Perelman3}. In this work, we discover for \eqref{eq: * kappa} with an affine coupling parameter, a coupled entropy which is indeed monotone and constant exactly on coupled gradient shrinking solitons; we refer the reader to \cref{Symmetries of the equation} for the precise definition.
\begin{theorem}
    \label{theorem W monotonicity}
    Let $\bigl(g(t), A(t)\bigr)$ be a solution of \eqref{eq: * kappa} with $\kappa(t) = t_0 - t$ for some $t_0 > 0$, and let $u$ be a positive smooth solution of the conjugate heat equation. Then, the \emph{coupled entropy}
    \begin{equation*}
        \mathcal W_\kappa(g,A,u) = \kappa\int_M \left(\R - \frac{\kappa}{4}|F|^2 + |\nabla\ln u|^2\right) u d\mu \,-\, \int_M \left(\ln u + \frac{n}{2}\ln(4\pi\kappa)\right)u d\mu
    \end{equation*}
    is non-decreasing for $0 \leq t \leq t_0$ and constant exactly on coupled gradient shrinking solitons.
\end{theorem}
\begin{remark}
    More generally, if $\bigl(g(t), A(t)\bigr)$ is a solution of \eqref{eq: * kappa} with coupling parameter $\kappa(t) = t_0 - \alpha t$ for some $t_0 > 0$ and $\alpha \in \nR$, a slight modification $\mathcal W_\kappa^\alpha(g, A, u)$ is monotone, see \cref{Coupled energy and entropy}.
\end{remark}

\subsection*{Further related work}
Ricci flow coupled to harmonic map heat flow was introduced by List \cite{List} in the case where the map is a scalar function and by Buzano \cite{Buzano2012} for more general target spaces, see also the works \cite{LottSesum,GuoHuangPhong,BuzanoRupfin,DiMatteo,Johne}. In particular, they established interior curvature estimates, and monotonicity of a coupled entropy. Similarly to us, the showed that the coupled flow exists as long as the Riemann curvature of $(M,g(t))$ is bounded. Their case is simpler though, since the evolution equation of the energy density $|\nabla u|^2$, which is analogue to $|F|^2$ in our work, does not involve the Riemann curvature $|\Rm|$. 

Generalized Ricci flow is another prominent extension of the Ricci flow, where one studies the coupled evolution of a Riemannian metric and a 2-form. However, we highlight that unlike for \eqref{eq: * kappa}, the evolution equation of the 2-form is linear. In such context, interior curvature estimates similar to the ones in the present paper were derived by Li \cite{LiGRFestimates}, and a monotone coupled entropy was only recently discovered by Streets \cite{StreetsMonotonicity} in the presence of a so called \emph{dilaton field}.

All these flows are examples of super Ricci flows, see \cite{ToppingMcCann,SturmSRF,BamlerSRF,FlameHupp}.

\subsection*{Outline}
The paper is structured as follows. In \cref{Preliminaries} we review background material, fix our conventions, and compute the basic identities and evolution equations. In \cref{Symmetries of the equation} we define coupled solitons and study a non-trivial explicit example. In \cref{Coupled energy and entropy} we introduce a coupled energy and entropy analogous to those of Perelman, and compute their evolution along \eqref{eq: * kappa}. \cref{Theorem preserved integral inequality} and \cref{theorem W monotonicity} are just consequences. \cref{A priori estimates I} are devoted to the main a priori estimates that lead to the proof of \cref{Theorem Short Time}. Finally, in \cref{Finite-time singularities} we study finite-time singularities and complete the proof of \cref{Theorem A}.

\subsection*{Acknowledgments}
The author is grateful to Gerhard Huisken for many helpful discussions regarding this work. The author was supported by the DAAD program \emph{Forschungsstipendien für Doktoraninnen und Dokoranden} during a visit to University of California Irvine and University of Wisconsin Madison. He wishes to thank Jeff Streets and Alex Waldron for their hospitality and interest in this work.

\section{Preliminaries}
\label{Preliminaries}
Let $E \to M$ be a vector bundle of rank $r$ over a closed $n$ dimensional manifold. Let $G$ be the structure group of $E$ and denote by $\mathscr G_E$ and $\mathfrak g_E$ the bundle of \emph{gauge transformations} and \emph{local gauge transformations} of $E$ respectively. In the general case $G = \GL(r)$, these are simply $\mathscr G_E = \Aut E$ and $\mathfrak g_E = \End E$, but otherwise the fibers of $\mathscr G_E$ and $\mathfrak g_E$ are isomorphic to $G$ and its Lie algebra $\mathfrak g$, respectively. For example, if $E$ is endowed with a fiber-wise inner product, then $G \subset O(r)$ and the fibers of $\mathfrak g_E$ are identified with skew-symmetric matrices.

We use the notation
\begin{equation*}
    \Omega^k(E) \coloneqq \big\{\textit{smooth sections of  } E \otimes \underbrace{T^*M \wedge \ldots \wedge T^*M}_{k-times}\bigr\}
\end{equation*}
for the space of $E$-valued $k$-forms. Given local coordinates $\{x_i\}_{i=1}^n$ on $M$ and a local frame $\{e_\alpha\}_{\alpha=1}^r$ of $E$, we may write any $\psi \in \Omega^k(E)$ locally like
\begin{equation*}
    \psi =  \frac{1}{k!} \, \psi^\alpha_{i_1 \ldots i_k} e_\alpha \otimes dx_{i_1} \wedge \ldots \wedge dx_{i_k}
\end{equation*}
for some unique components $\psi^\alpha_{i_1 \ldots i_k} \in C^\infty(M)$. We adopt Einstein's summation convention and the normalization $dx_i \wedge dx_j = dx_i \otimes dx_j - dx_j \otimes dx_i$. Given a Riemannian metric $g$ on $M$, we use the pointwise inner product induced on fibers of $\otimes^kTM$, not $\wedge^kTM$. This amounts to
\begin{equation*}
    \langle\omega,\eta\rangle = g^{ip}g^{jq} \omega_{ij} \eta_{pq} = \omega_{ij}\eta^{ij}
\end{equation*}
for any $\omega, \eta \in \Omega^2$.

\subsection{Connections and curvature}
\label{Connections and curvature}
Every fiber of $E$ is isomorphic to $\nR^r$, but unless the bundle is trivial, such identification is not canonical. A \emph{connection} on $E$ is a mean of identification between different fibers, and it is described by a map
\begin{equation*}
    \nabla : \Omega^0(E) \to \Omega^1(E)
\end{equation*}
satisfying the product rule
\begin{equation}
    \label{eq: connection def}
    \nabla(f\psi) = df \otimes \psi + f \nabla\psi
\end{equation}
for any $f \in C^\infty(M)$ and $\psi \in \Omega^0(E)$. If the vector bundle is endowed with a fiber-wise inner product, we require the connection to satisfy the compatibility condition
\begin{equation*}
    d\langle\psi,\phi\rangle = \langle\nabla \psi,\phi\rangle + \langle\psi,\nabla\phi\rangle \,.
\end{equation*}
Given local coordinates $\{x_i\}_{i=1}^n$ on $M$ and a local frame $\{e_\alpha\}_{\alpha=1}^r$ of $E$, we have 
\begin{equation}
    \label{eq: connection local}
    \nabla\psi = \bigl(\partial_i\psi^\alpha + A^\alpha_{i\beta}\psi^\beta\bigr) e_\alpha \otimes dx_i
\end{equation}
where 
\begin{equation*}
    A^\alpha_{i\beta} \coloneqq \langle e_\alpha, \nabla_i e_\beta\rangle
\end{equation*}
are the components of the \emph{connection one-form} of $\nabla$ in this trivialization. If the frame is orthonormal, the compatibility condition implies that $A^\alpha_{i\beta}$ is skew-symmetric in the bundle indices, and the connection one-form can locally be regarded as a one-form with values in $\mathfrak g$. The lack of tensoriality apparent in \eqref{eq: connection def}, though, tells us that $A$ is not a global section of $\mathfrak g_E \otimes T^*M$. Instead, given a gauge transformation $\sigma \in \Omega(\mathscr G_E)$, the connection one-form in the frame $\bar e_\alpha = \sigma^\beta_{\;\;\alpha} e_\beta$ reads like
\begin{equation}
    \label{eq: gauge transformation connection}
    \bar A^\alpha_{i\beta} = \sigma^\alpha_{\;\;\mu} A^{\mu}_{i\nu} \sigma_\beta^{\;\;\nu} - \partial_i\sigma^{\alpha}_{\;\;\mu} \sigma_{\beta}^{\;\;\mu}
\end{equation}
with respect to the original frame, where $\sigma^{\;\;\alpha}_{\beta} \coloneqq (\sigma^{-1})^\alpha_{\;\;\beta}$. On the other hand, we observe that the difference of two connections is indeed tensorial and thus the space $\mathscr A_E$ of connections on $E$ is an affine space modeled on $\Omega^1(\mathfrak g_E)$. This means that given\footnote{It follows from \eqref{eq: connection def} that any convex combination of connections is again a connection, which combined with a simple partition of unity argument, implies that $\mathscr A_E$ is always non-empty.} a reference connection $\nabla_0$, any other connection in $\mathscr A_E$ can be uniquely written like $\nabla_0 + A$ for some $A \in \Omega^1(\mathfrak g_E)$.

If $E$ is a trivial bundle, we may naturally choose $\nabla_0 = d$. This is precisely what happens in a trivialization, where
\begin{equation*}
    \nabla\psi = d\psi + A\cdot\psi
\end{equation*}
as seen in \eqref{eq: connection local}. It is common practice to suppress the bundle indices in this way and write all expressions in a trivialization. For instance, we write \eqref{eq: gauge transformation connection} simply as
\begin{equation*}
    \sigma(A) = \sigma \cdot A \cdot \sigma^{-1} - d\sigma \cdot \sigma^{-1} \,.
\end{equation*}
Henceforth, we will often adopt the same habit and we will refer to any element of $\mathscr A_E$ simply by $A$.

If $M$ is equipped with a Riemannian metric, we also denote by $\nabla$ the covariant derivative on sections of any tensor bundle constructed from $TM$ and $E$, induced by the Levi--Civita connection and $A$ in the obvious way. For example, for $\omega \in \Omega^1(\End E)$ we have
\begin{equation*}
    \nabla_i \omega_{j\beta}^\alpha = \partial_i \omega_{j\beta}^\alpha - \Gamma_{ij}^k \omega_{k\beta}^\alpha + A^\alpha_{i\gamma} \omega^\gamma_{j\beta} - A^\gamma_{i\beta} \omega_{j\gamma}^\alpha \,,
\end{equation*}
or suppressing the bundle indices,
\begin{equation*}
    \nabla_i \omega_j = \partial_i \omega_j - \Gamma^k_{ij} \omega_k + [A_i, \omega_j] \,.
\end{equation*}
Additionally, a connection on $E$ extends to a \emph{covariant exterior derivative}
\begin{equation}
    \label{eq: covariant exterior derivative}
    D : \Omega^k(E) \to \Omega^{k+1}(E)
\end{equation}
by the product rule 
\begin{equation*}
    D(f\psi) = df \wedge \psi + f D\psi .
\end{equation*}
This is nothing but the totally antisymmetric part of $\nabla$ and it simply amounts to
\begin{equation*}
    D\psi = \nabla_i\psi_{i_1 \ldots i_k} \; dx_i \wedge dx_{i_1} \wedge \ldots \wedge dx_{i_k} \,.
\end{equation*}
An important observation is that due to the torsion-freeness of the Levi--Civita connection, the Christoffel symbol terms cancel out and we obtain
\begin{equation*}
    D\psi = d\psi + A \wedge \psi
\end{equation*}
for any $\psi \in \Omega^k(E)$ and
\begin{equation*}
    D\omega = d\omega + A \wedge \omega + (-1)^{k+1}\omega \wedge A
\end{equation*}
for any $\omega \in \Omega^k(\End E)$. In particular, the covariant exterior derivative is independent of the Riemannian structure of $M$. Unlike with the usual exterior derivative, though, \eqref{eq: covariant exterior derivative} is not a chain complex and the obstruction
\begin{equation*}
    F \coloneqq D^2 : \Omega^0(E) \to \Omega^2(E)
\end{equation*}
is precisely the curvature of $A$. It is easy to check that $F$ is tensorial and thus a section of $\Omega^2(\mathfrak g_E)$ with coefficients
\begin{equation}
    \label{eq: curvature components}
    F_{ij} = \partial_iA_j - \partial_jA_i + [A_i,A_j] \,.
\end{equation}
In fact, $F \equiv 0$ if and only if there exist local frames such that $A$ vanishes identically. If we write $\nabla = \nabla_0 + A$ for some $A \in \Omega^1(\mathfrak g_E)$, then one can compute
\begin{equation*}
    F(\nabla) = F(\nabla_0) + D_0A + A \wedge A \,,
\end{equation*}
which recovers the local expression \eqref{eq: curvature components} by inserting $\nabla_0 = d$.

With this, we have the commutation identity\footnote{We adopt R. Hamilton's convention for the Riemannian curvature \cite{Hamilton82}, which amounts to $\Rm(X,Y)Z = \nabla_X\nabla_YZ - \nabla_Y\nabla_XZ - \nabla_{[X,Y]}Z$ and $\Rm(X,Y,Z,W) = \langle\Rm(X,Y)W,Z\rangle$.}
\begin{equation*}
    [\nabla_i, \nabla_j] \, \omega^{k\;\alpha}_{\;\;l\;\;\beta} = \R_{ij\;\;p}^{\;\;\;k} \omega^{p\;\alpha}_{\;\;l\;\;\beta} - \R_{ij\;\;l}^{\;\;\;p} \omega^{k\;\alpha}_{\;\;p\;\;\beta} + F_{ij\;\;\gamma}^{\;\;\;\alpha} \omega^{k\;\gamma}_{\;\;l\;\;\beta} - F_{ij\;\;\beta}^{\;\;\;\gamma} \omega^{k\;\alpha}_{\;\;l\;\;\gamma}
\end{equation*}
for a section $\omega$ of $\End(TM) \otimes \End(E)$, and likewise for sections of other bundles. Additionally, there holds the second Bianchi identity
\begin{equation}
    \label{eq: 2 Bianchi}
    DF = 0 \,,
\end{equation}
which is written more commonly in coordinates as
\begin{equation*}
    \nabla_iF_{jk} + \nabla_jF_{ki} + \nabla_kF_{ij} = 0 \,.
\end{equation*}

\subsection{Self-duality in four dimensions}
\label{Self-duality in four dimensions}

In four dimensions the Hodge star $* : \Omega^2 \to \Omega^{2}$ on 2-forms (with values in any bundle) satisfies $*^2 = 1$, and thus the space of 2-forms has an orthogonal decomposition
\begin{equation}
    \label{eq: SD ASD decomposition}
    \Omega^2 = \Omega^{2}_+ \oplus \Omega^{2}_-
\end{equation}
of the eigenspaces of $*$ with eigenvalues +1 and the -1. Any 2-form can consequently be written as $\omega = \omega^+ + \omega^-$, where
\begin{equation*}
    \omega^+ = \frac{1}{2}(\omega + *\omega) \; \in \Omega^{2}_+ \quad \text{and} \quad \omega^- = \frac{1}{2}(\omega - *\omega) \; \in \Omega^{2}_-
\end{equation*}
are called the \emph{self-dual} and \emph{anti-self-dual} parts of the 2-form. A connection is self-dual or anti-self-dual if its curvature 2-form satisfies
\begin{equation*}
    F^- = 0 \quad \text{or} \quad F^+ = 0
\end{equation*}
respectively. Self-dual and anti-self-dual connections are called \emph{instantons}. The $F = F^+ + F^-$ splitting of curvature implies that
\begin{equation*}
    Z_{ij} = \frac{1}{4}\bigl(|F^+|^2 + |F^-|^2\bigr) g_{ij} + 2 \, g^{pq} \langle F_{ip}^+, F_{jq}^-\rangle \,,
\end{equation*}
and therefore, a connection is an instanton if and only if the traceless part $\mathring Z$ of $Z$ vanishes; see \cite[Lemma A.4]{StreetsThesis} for example. Additionally, we can rewrite the Yang--Mills energy as
\begin{align*}
    \int_M |F|^2 \, d\mu &= \int_M \bigl(|F^+|^2 - |F^-|^2\bigr) \,d\mu + 2 \int_M |F^-|^2 \,d\mu
    \\
    &= \int_M \langle F, *F \rangle \,d\mu + 2 \int_M |F^-|^2 \,d\mu
    \\
    &= \int_M \tr F \wedge F + 2 \int_M |F^-|^2 \,d\mu \,,
\end{align*}
where the trace in the last line is over the bundle indices. If $M$ is a compact oriented manifold, Chern--Weil theory tells us that the first integrand is a topological invariant of the vector bundle, often called the \emph{instanton number}. We shall denote it by
\begin{equation*}
    \tau(E) \coloneqq \frac{1}{8\pi^2} \int_M \tr F \wedge F \;.
\end{equation*}
In the case of an $\SU(2)$ bundle, $\tau(E)$ is the second Chern number $c_2(E) \in \nZ$. We conclude that
\begin{equation}
    \label{eq: YM energy lower bound}
    \int_M |F|^2 \, d\mu = 8\pi^2 \tau(E) + 2 \int_M |F^-|^2 \,d\mu \geq 8\pi^2 \tau(E) \,.
\end{equation}
A similar formula holds with $F^+$ instead of $F^-$. We observe that the infimum of the Yang--Mills energy is a topological invariant, which is attained precisely at instantons. This is in parallelism with other variational problems in geometric analysis, such as the Willmore energy or the Weyl functional.

\subsection{Evolution equations and curvature identities}
\label{Evolution equations and curvature identities}
First, we collect the evolution equation of basic objects from Riemannian geometry. Following \cite{Hamilton82}, we denote by $d\mu$ the volume form of $g$ and define the quadratic curvature term
\begin{equation*}
    Q_{ijkl} \coloneqq \R_{ijpq}\R_{kl}^{\;\;\;pq} + 2 \R_{ipkq}\R_{\,\;j\;l}^{p\;q} - 2\R_{ipql}\R^{p\;\;q}_{\;jk} \,.
\end{equation*}
\begin{proposition}
    \label{prop: evol eq Riemannian}
    Under the evolution equation \eqref{eq: * kappa} the following hold.
    \begin{equation*}
        \frac{\partial}{\partial t} d\mu = \left(-\R + \frac{\kappa}{2}|F|^2\right) d\mu
    \end{equation*}
    \begin{align*}
        \frac{\partial}{\partial t}\R_{ijkl} &= \Delta\R_{ijkl} + Q_{ijkl} - \R_{pjkl}\R^p_{\;i} - \R_{ipkl}\R^p_{\;j} - \R_{ijpl}\R^p_{\;k} - \R_{ijkp}\R^p_{\;l}
        \\
        & + \frac{\kappa}{2}\Big(\R_{ijpl}Z^p_{\;k} + \R_{ijkp}Z^p_{\;l} - \nabla_i\nabla_kZ_{jl} + \nabla_i\nabla_lZ_{jk} + \nabla_j\nabla_kZ_{il} - \nabla_j\nabla_lZ_{ik}\Big)
    \end{align*}
    \begin{align*}
        \frac{\partial}{\partial t}\R_{ij} &= \Delta\R_{ij} + 2\R_{ipjq}\R^{pq} - 2\R_{ip}\R_j^{\;p}
        \\
        &+ \frac{\kappa}{2}\Bigl(\R_{ip}Z_{j}^{\;p} - \R_{ipjq}Z^{pq} - \nabla_i\nabla_j|F|^2 + \nabla_i\nabla^pZ_{pj} + \nabla^p\nabla_jZ_{jp} - \Delta Z_{ij}\Bigr)
    \end{align*}
    \begin{align*}
        \frac{\partial}{\partial t}\R = \Delta\R + 2|\Ric|^2 - \kappa \langle\Ric,Z\rangle - \kappa\Delta|F|^2 + \kappa \Div^2Z
    \end{align*}    
\end{proposition}
\begin{proof}
    The evolution of the volume form follows from Jacobi's formula for the derivative of a determinant applied to $d\mu = \sqrt{\det g}\; dx$. The evolution of curvature follows from the well-known formula
    \begin{align*}
        \frac{\partial}{\partial t}\R_{ijkl} &= -\frac{1}{2}\bigl(\nabla_i\nabla_kh_{jl} - \nabla_i\nabla_lh_{jk} - \nabla_j\nabla_kh_{il} + \nabla_j\nabla_lh_{ik}\bigr)
        \\
        &\quad + \frac{1}{2}\bigl(\R_{ijpl}h^p_{\;k} + \R_{ijkp}h^p_{\;l}\bigr)
    \end{align*}
    for a smooth one-parameter family of metrics evolving with speed $h_{ij}$, together with the equation
    \begin{align*}
        \Delta\R_{ijkl} &= \nabla_i\nabla_k\R_{jl} - \nabla_i\nabla_l\R_{jk} - \nabla_j\nabla_k\R_{il} + \nabla_j\nabla_l\R_{ik} 
        \\
        &- Q_{ijkl} + \R_{pjkl}\R^p_{\;i} + \R_{ipkl}\R^p_{\;j} \,.
    \end{align*}
    due to Bianchi identities. See \cite[Lemma 7.2]{Hamilton82} or \cite[§ 2.3]{BrendleBook} for instance. The evolution for the Ricci and scalar curvature now follow from
    \begin{equation*}
        \frac{\partial}{\partial t}\R_{ij} = 2\R_{ipjq}\R^{pq} - \kappa \R_{ipjq}Z^{pq} + g^{pq} \frac{\partial}{\partial t}\R_{ipjq}
    \end{equation*}
    and
    \begin{equation*}
        \frac{\partial}{\partial t}\R = 2|\Ric|^2 - \kappa \langle\Ric, Z\rangle + g^{ij} \frac{\partial}{\partial t}\R_{ij} \,,
    \end{equation*}
    together with the fact that $g^{pq}Q_{ipjq} = 2\R_{ipjq}\R^{pq}$ due to Bianchi.
\end{proof}
\begin{proposition}
    \label{prop: evol eq bundle 1}
    Under the evolution equation \eqref{eq: * kappa} the following hold. 
    \begin{align*}
        \frac{\partial}{\partial t} F_{ij} = \Delta F_{ij} + 2 \R_{ipjq} F^{pq} + \R_i^{\;p} F_{jp} - \R_j^{\;p} F_{ip} - 2[F_{ip},F_j^{\;p}]
    \end{align*}
    \begin{align*}
        \frac{\partial}{\partial t} |F|^2 = \Delta |F|^2 - 2 |\nabla F|^2 - 2\kappa \, |Z|^2 + 2 \R_{ijkl} \langle F^{ij}, F^{kl} \rangle - 4 \langle F_i^{\;p}, [F^{iq}, F_{pq}] \rangle 
    \end{align*}
\end{proposition}
\begin{proof}
    From the variation $\partial_tF = D\partial_tA$ of the curvature 2-form, it follows that
    \begin{equation}
        \label{eq: evol F}
        \frac{\partial F}{\partial t} = -DD^*F
    \end{equation}
    along \eqref{eq: * kappa}. In light of the Bianchi identity \eqref{eq: 2 Bianchi}, we see that
    \begin{equation*}
        \frac{\partial F}{\partial t} = - \Delta_A F \,,
    \end{equation*}
    where $\Delta_A \coloneqq DD^* + D^*D$ is the Hodge Laplacian induced by $A$, and the first evolution equation follows from the Weitzenböck identity
    \begin{align*}
        -\Delta_A\omega_{ij} &= \Delta\omega_{ij} - [F_{ip},\omega_{j}^{\;p}] + [F_{jp},\omega_{i}^{\;p}] + 2\R_{ipjq}\omega^{pq} + \R_i^{\;p} \omega_{jp} - \R_j^{\;p} \omega_{ip}
    \end{align*}
    for sections $\omega \in \Omega^2(\End E)$ (see \cite[Appendix II]{LawsonBook} or \cite[§1.5]{Waldron2016}). The second equation follows from inserting the first one in
    \begin{equation}
        \label{eq: intermeq1}
        \frac{\partial}{\partial t} |F|^2 = 4\langle\Ric, Z\rangle + 2\kappa |Z|^2 + 2 \Bigl\langle \frac{\partial}{\partial t}F_{ij}, F^{ij}\Bigr\rangle
    \end{equation}
    and using the Bochner-type identity
    \begin{equation*}
        \Delta|F|^2 = 2\langle\Delta F_{ij}, F^{ij}\rangle + 2|\nabla F|^2
    \end{equation*}
    and the equality
    \begin{equation*}
        2\R_{ijkl} \langle F^{ik},F^{jl}\rangle = \R_{ijkl} \langle F^{ij},F^{kl}\rangle
    \end{equation*}
    due to the Bianchi identity.
\end{proof}
Instead of using the Weitzenböck identity, one can compute the evolution equation of $|F|^2$ using the following divergence identities:
\begin{lemma}
    \label{lemma: divergence identities}
    \begin{equation*}
        \nabla^pZ_{pi} = \frac{1}{4}\nabla_i|F|^2 + \langle \nabla^p F_{pq}, F_i^{\;q} \rangle \,.
    \end{equation*}
    \begin{equation*}
        \Div^2Z = \frac{1}{4}\Delta|F|^2 + |D^*F|^2 - \frac{1}{2}\langle F, DD^*F\rangle \,.
    \end{equation*}
\end{lemma}
\begin{proof}
    If we take a divergence of $Z$,
    \begin{align*}
        \nabla^i Z_{ij} &= \langle \nabla^i F_{ip}, F_j^{\;p} \rangle + \langle F^{ip}, \nabla_i F_{jp} \rangle
    \end{align*}
    and use the Bianchi identity to write the second term as
    \begin{align*}
        \langle F^{ip}, \nabla_i F_{jp} \rangle &= \langle F^{ip}, \nabla_j F_{ip} \rangle + \langle F^{ip}, \nabla_p F_{ji} \rangle \\
        &= \frac{1}{2} \nabla_j |F|^2 - \langle F^{pi}, \nabla_p F_{ji} \rangle = \frac{1}{4} \nabla_j |F|^2 \,,
    \end{align*}
    the first identity follows. Taking yet another divergence,
    \begin{align*}
        \nabla^i \nabla^j Z_{ij} &= \frac{1}{4} \Delta |F|^2 + \langle \nabla^iF_{ik}, \nabla^j F_j^{\;k} \rangle + \langle F_{ik}, \nabla^i \nabla^j F_{j}^{\;k} \rangle \\
        &= \frac{1}{4} \Delta |F|^2 + |D^*F|^2 + \frac{1}{2}\langle F_{ik}, \nabla^i \nabla^j F_{j}^{\;k} - \nabla^k \nabla^j F_{j}^{\;i}\rangle \,,
    \end{align*}
    as desired.
\end{proof}
\begin{proposition}
    \label{prop: evol eq bundle 2}
    Under the evolution equation \eqref{eq: * kappa} there holds 
    \begin{equation*}
        \frac{\partial}{\partial t}|F|^2 = - \Delta|F|^2 - 4|D^*F|^2 - 2\kappa|Z|^2 + 4\langle\Ric, Z\rangle + 4\Div^2Z \,.
    \end{equation*}
    If the manifold is compact, then
    \begin{align*}
        \frac{d}{dt}\int_M |F|^2 &= -4 \int_M |D^*F|^2 - 2\kappa \int_M |\mathring Z|^2 + 4\int_M \langle\mathring\Ric, \mathring Z\rangle 
        \\
        &\quad + \frac{4-n}{2n}\left(2\int_M \R|F|^2 - \kappa\int_M |F|^4\right) \,,
    \end{align*}
    where $\mathring Z$ is the traceless part of $Z$.
\end{proposition}
\begin{proof}
    Inserting the evolution of curvature of \eqref{eq: evol F} in equation \eqref{eq: intermeq1} and combining with the second identity of \cref{lemma: divergence identities}, we deduce
    \begin{align*}
        \frac{\partial}{\partial t}|F|^2 &= 4 \langle\Ric, Z\rangle - 2\kappa\,|Z|^2 - 2\langle DD^*F, F\rangle
        \\
        &= 4 \langle\Ric, Z\rangle - 2\kappa\,|Z|^2 + 4\Div^2Z - \Delta|F|^2 - 4|D^*F|^2 \,, 
    \end{align*}
    which yields the first identity. Integrating and using the evolution of the volume form of \cref{prop: evol eq Riemannian}, we deduce
    \begin{align*}
        \frac{d}{dt}\int_M |F|^2 &= -4 \int_M |D^*F|^2 - 2\kappa \int_M |Z|^2 + 4\int_M \langle\Ric, Z\rangle - \int_M \R|F|^2 + \frac{\kappa}{2} \int_M |F|^4
        \\
        &= -4 \int_M |D^*F|^2 - 2\kappa \int_M |\mathring Z|^2 + 4\int_M \langle\mathring\Ric, \mathring Z\rangle 
        \\
        &\quad + 4\left(\frac{1}{n}-\frac{1}{4}\right)\int_M \R|F|^2 + 2\kappa\left( \frac{1}{4} - \frac{1}{n}\right)\int_M |F|^4 \,.
    \end{align*}
\end{proof}
We observe that the Yang--Mills energy is not necessarily monotone under the coupled flow, despite having a new good term due to the coupling. This is due to the deformation of the volume element. However, in dimension four, the obstruction is small whenever $g$ is close to being Einstein or $A$ close to being an instanton, i.e. $\mathring Z \equiv 0$.

\section{Symmetries of the equation}
\label{Symmetries of the equation}
An important class of solutions to geometric evolution equations is that generated by the symmetries, usually called \textit{self-similar solutions} or \textit{solitons}. In this section we define self-similar solutions to \eqref{eq: * kappa} and provide a non-trivial example.

\subsection{Self-similar solutions}
\label{Self-similar solutions}
We define self-similar solutions to \eqref{eq: * kappa} as follows, c.f. \cite[Definition 3.6]{StreetsThesis} and \cite[Definition 2.1]{YoungNilpotent}.
\begin{definition}
    Let $E \to M$ be a vector bundle with structure group $G$. A tuple $(g, A, X,\omega)$ consisting of a Riemannian metric on $M$, a connection on $E$, a vector field on $M$ and a section of $\mathfrak g_E$, is called a \emph{coupled soliton} if there exists some $\rho \in \nR$ such that
    \begin{align}
        2\Ric - Z &= \rho g - \mathcal L_Xg
        \\
        D^*F &= \nabla\omega - \mathcal L_XA \;.
    \end{align}
    The coupled soliton is said to be \emph{shrinking}, \emph{steady} or \emph{expanding} depending on whether $\rho > 0$, $\rho = 0$ or $\rho < 0$ respectively. Moreover, it is called \emph{complete} if $(M,g)$ as well as $X$ and $\omega$ are complete.
\end{definition}
\begin{remark}
    Observe that if $(g,A,X,\omega)$ is a coupled soliton with parameter $\rho$, then $(\lambda g,A,\lambda^{-1}X,\lambda^{-1}\omega)$ is a coupled soliton with parameter $\lambda^{-1}\rho$ for any $\lambda > 0$.
\end{remark}
Coupled solitons naturally generate solutions to \eqref{eq: * kappa} via scalings, diffeomorphisms, and gauge transformations:
\begin{theorem}
    \label{thm: solitons generate solutions}
    Let $E \to M$ be a vector bundle over a closed manifold and let $(g, A, X, \omega)$ be a complete coupled soliton. Define $\rho(t) \coloneqq 1-\rho t$, $X(t) \coloneqq \rho^{-1}(t)X$ and $\omega(t) \coloneqq \rho^{-1}(t)\omega$. Let $\varphi_t$ and $\sigma_t$ be the smooth one-parameter family of diffomorphisms and gauge transformations generated by $X(t)$ and $\omega(t)$ respectively. Then, 
    \begin{equation*}
        g(t) \coloneqq \rho(t) \varphi_t^*g \,, \quad A(t) \coloneqq \varphi_t^*\sigma_t(A)
    \end{equation*}
    satisfies \eqref{eq: * kappa} with coupling parameter $\kappa(t) = \rho(t)$.
\end{theorem}
\begin{proof}
    A simple computation (see \cite[Lemma 2.3.1, Lemma 2.3.2]{RoyoThesis}) yields
    \begin{align*}
        \frac{\partial g}{\partial t}(t) &= -\varphi_t^*\bigl(\rho g - \rho(t)\mathcal L_{X(t)} g\bigr) = -\varphi_t^*\bigl(\rho g - \mathcal L_X g\bigr) 
        \\
        &= - 2\Ric_{\varphi_t^*g} + Z_{\varphi_t^*g,\varphi_t^*A} = -2\Ric_{g(t)} + \rho(t) Z_{g(t),A(t)} \,,
    \end{align*}
    and 
    \begin{align*}
        \frac{\partial A}{\partial t}(t) &= - \sigma_t\bigl(\varphi_t^*\bigl(\nabla\omega(t) - \mathcal L_{X(t)}A\bigr)\bigr) = -\rho^{-1}(t) \sigma_t\bigl(\varphi_t^*\bigl(\nabla\omega - \mathcal L_XA\bigr)\bigr) 
        \\
        &= -\rho^{-1}(t) \sigma_t\bigl(\varphi_t^*\bigl(D^*_{g,A}F_A\bigr)\bigr) = -D^*_{g(t),A(t)}F_{A(t)} \,.
    \end{align*}
\end{proof}

\subsection{An explicit example}
\label{An important example}
Let $M = \nS^4_a \coloneq \{x \in \nR^5 \,:\, |x| = a\}$ for some radius $a > 0$ and $E$ an $\SU(2)$ vector bundle over $M$ with $c_2(E) = 1$. In \emph{normalized stereographic coordinates}\footnote{By \textit{normalized stereographic coordinates} we mean coordinates induced by the stereographic projection on the unit sphere.} around the north pole, the metric reads
\begin{equation}
    \label{eq: round metric stereograph}
    g(x) = \frac{4a^2}{\bigl(1+|x|^2\bigr)^2} dx_i \otimes dx_i \,.
\end{equation}
An advantage of stereographic coordinates is that it allows us to work in $\nR^4$, where quaternions are naturally defined. A \emph{quaternion} is a formal linear combination
\begin{equation*}
    x = \sum_{i=1}^4 x_i \sigma_i = x_1 \sigma_1 + x_2 \sigma_2 + x_3 \sigma_3 + x_4 \sigma_4 \,, \qquad x_i \in \nR \,,
\end{equation*}
where the generators $\{\sigma_1, \sigma_2, \sigma_3, \sigma_4\}$ satisfy a product rule\footnote{We follow the convention in \cite{SchlatterStruweShadi} and \cite{AtiyahBook}. For computations we recall that the Levi--Civita symbol satisfies
\begin{equation*}
    \epsilon_{ijpq}\epsilon_{klpq} = 2(\delta_{ik}\delta_{jl} - \delta_{il}\delta_{jk}) \quad \text{and} \quad \epsilon_{ipqs}\epsilon_{jpqs} = 6\delta_{ij} \,.
\end{equation*}
}
\begin{equation}
    \label{eq: quaternion product}
    \sigma_i\sigma_j = \delta_{1i}\sigma_j + \delta_{1j}\sigma_i - \delta_{ij}\sigma_1 + \epsilon_{1ijk}\sigma_k \,.
\end{equation}
This amounts to the relations
\begin{equation*}
    \label{eq: quaternion rule 1}
    -\sigma_1^2 = \sigma_2^2 = \sigma_3^2  = \sigma_4^2 = - \sigma_1 \,,
\end{equation*}
\begin{equation*}
    \label{eq: quaternion rule 2}
    \sigma_2 \sigma_3 = -\sigma_3 \sigma_2 = \sigma_4 \,, \qquad \sigma_4 \sigma_2 = -\sigma_2 \sigma_4 = \sigma_3 \,, \qquad \sigma_3 \sigma_4 = -\sigma_4\sigma_3 = \sigma_2 \,.
\end{equation*}
The product \eqref{eq: quaternion product} extends by linearity to a non-commutative product between arbitrary quaternions, endowing the linear space $\nR^4$ with a structure of an algebra over $\nR$ with multiplicative identity $\sigma_1$. Such algebra is denoted by $\nH$. The conjugate of a quaternion $x \in \nH$ is defined by
\begin{equation*}
    \bar x = x_1\sigma_1 - x_2 \sigma_2 - x_3 \sigma_3 - x_4 \sigma_4 \,.
\end{equation*}
It is immediate to see that it satisfies $\overline{xy} = \bar y \bar x$ (anti-involution) and therefore it induces a norm
\begin{equation*}
    |x|^2 = x \bar x = x_1^2 + x_2^2 + x_3^2 + x_4^2
\end{equation*}
as well as a multiplicative inverse $x^{-1} = \bar x/|x|^2$. In particular, we have that if $x$ and $y$ are quaternions of unit norm, then 
\begin{equation*}
    |xy|^2 = xy\overline{xy} = x y \bar y \bar x = x \bar x = 1 \,.
\end{equation*}
Consequently, the space $\nU$ of unit quaternions, which is naturally identified with $\nS^3$, forms a non-abelian group. In other words, $\nS^3$ inherits a group structure from $\nH$, just the same way $\nS^1$ does it from $\nC$. In analogy with complex numbers, we call $x_1$ the \emph{real} and $x - x_1\sigma_1$ the \emph{imaginary} part of $x$, and by taking $x_3 = x_4 = 0$ we can regard $\nC = \{x_1 \sigma_1 + x_2\sigma_2 \; : \; x_1, x_2 \in \nR \}$ as a linear subspace of $\nH$. The real and imaginary parts can also be written as
\begin{equation*}
    \text{Re}(x) = \frac{1}{2}(x + \bar x) \,, \qquad \text{Im}(x) = \frac{1}{2}(x - \bar x) \,.
\end{equation*}
Furthermore, every quaternion can be written as $x = z_1 + z_2 \sigma_3$ with
\begin{equation*}
    z_1 = x_1\sigma_1 + x_2 \sigma_2 \qquad \text{and} \qquad z_2 = x_3\sigma_1 + x_4 \sigma_2 \,,
\end{equation*}
allowing us to identify $\nH$ with $\nC^2$ (as vector spaces). Now fix a quaternion $y = w_1 + w_2 \sigma_2$, with $w_1, w_2 \in \nC$, and compute the quaternion (right) multiplication
\begin{align*}
    x y &= (z_1 + z_2 \sigma_2)(w_1 + w_2 \sigma_2) = z_1 w_1 - w_2 \bar w_2 + (z_1w_2 + z_2 \bar w_1) \sigma_2 \,.
\end{align*}
Thus, if we understand $x \in \nC^2$ as a complex-valued column vector, the quaternion multiplication can be written as
\begin{equation*}
    xy = \begin{pmatrix}
    z_1 \\
    z_2
\end{pmatrix}
\begin{pmatrix}
    w_1       & w_2 \\
    -\bar w_2 & \bar w_1
\end{pmatrix}.
\end{equation*}
In particular if we substitute $y$ with the basis elements $\{\sigma_1, \sigma_2, \sigma_3, \sigma_4\}$, the matrix representation of their quaternion  multiplication becomes
\begin{equation*}
    \sigma_1 =
    \begin{pmatrix}
        1 & & 0 \\
        0 & & 1
    \end{pmatrix}
    ,\qquad
    \sigma_2 =
    \begin{pmatrix}
        i & 0 \\
        0 & -i
    \end{pmatrix}
    ,\qquad
    \sigma_3 =
    \begin{pmatrix}
        0  & 1 \\
        -1 & 0
    \end{pmatrix}
    ,\qquad 
    \sigma_4 =
    \begin{pmatrix}
        0 & & i \\
        i & & 0
    \end{pmatrix} \,.
\end{equation*}
Therefore, we may regard $\nH$ as a 4 (real) dimensional subalgebra of $\GL(2,\nC)$ spanned by $\{\sigma_1, \sigma_2, \sigma_3, \sigma_4\}$. In particular, the group of unit quaternions is identified with $\SU(2)$ and the Lie algebra $\mathfrak{su}(2)$ with purely imaginary ($x_1 = 0$) quaternions, that is,
\begin{equation*}
    \nU \cong \SU(2) \quad \text{and} \quad \text{Im}\nH \cong \frak{su}(2) \,.
\end{equation*}
The basis $\{\sigma_2, \sigma_3, \sigma_4\}$ of $\frak{su}(2)$ is orthonormal with respect to $\langle A, B \rangle = \tfrac{1}{2}\tr(A \cdot B^*)$ and commutation of quaternions
\begin{equation*}
    [\sigma_i,\sigma_j] = \sigma_i\sigma - \sigma_j\sigma_i = 2 \, \epsilon_{1ijk}\sigma_k
\end{equation*}
is the Lie bracket on $\mathfrak{su}(2)$. On the other hand, the space of self-dual 2-forms $\Omega^2_+$ in $\nR^4$ is a three dimensional linear space with basis
\begin{align*}
    \omega_1 &= dx_1 \wedge dx_2 + dx_3 \wedge dx_4
    \\
    \omega_2 &= dx_1 \wedge dx_3 - dx_2 \wedge dx_4
    \\
    \omega_3 &= dx_1 \wedge dx_4 + dx_2 \wedge dx_3 \,.
\end{align*}
Similarly, the 2-forms
\begin{align*}
    \bar\omega_1 &= dx_1 \wedge dx_2 - dx_3 \wedge dx_4
    \\
    \bar\omega_2 &= dx_1 \wedge dx_3 + dx_2 \wedge dx_4
    \\
    \bar\omega_3 &= dx_1 \wedge dx_4 - dx_2 \wedge dx_3 \,.
\end{align*}
form a basis of the space $\Omega_-^2$ of anti-self-dual 2-forms. Now, we may employ the quaternion algebra on $\nR^4$ and define the \emph{quaternion differentials} 
\begin{equation*}
    dx = dx_1\sigma_1 + dx_2 \sigma_2 + dx_3 \sigma_3 + dx_4 \sigma_4 \,, \quad d\bar x = dx_1\sigma_1 - dx_2 \sigma_2 - dx_3 \sigma_3 - dx_4 \sigma_4 \,,
\end{equation*}
which can simply be regarded as quaternion valued 1-forms. A direct computation yields
\begin{equation}
    \label{eq: wedge dx dxbar}
    dx \wedge d\bar x = -2 \omega_1 \sigma_2 - 2 \omega_2 \sigma_3 - 2\omega_3 \sigma_4 \,, \quad d\bar x \wedge dx = 2 \bar\omega_1 \sigma_2 + 2 \bar\omega_2 \sigma_3 + 2\bar\omega_3 \sigma_4 \,,
\end{equation}
showing that $dx \wedge d\bar x$ and $d\bar x \wedge dx$ are self-dual and anti-self-dual 2-forms respectively, with values in $\text{Im}(\nH) \cong \mathfrak{su}(2)$. In other words, they are elements of $\Omega^2_+(\mathfrak{su}(2))$ and $\Omega^2_-(\mathfrak{su}(2))$ respectively.

Using stereographic coordinates and quaternions, one constructs for each $\lambda \in (0,1]$ the connection 1-form
\begin{equation}
    \label{eq: 5-family instantons}
    A^{\lambda}(x) = \text{Im}\left(\frac{\bar x \, dx}{\lambda^2 + |x|^2}\right) \,.
\end{equation}
By the discussion above, $A^\lambda \in \Omega^1(\mathfrak{su}(2))$ and as such, it define an $\SU(2)$-connection on $\nS^4_a \setminus \{p\}$. One can compute its curvature
\begin{equation}
    \label{eq: lambda instanton curvature}
    F^\lambda(x) = \frac{\lambda^2}{\bigl(\lambda^2 + |x|^2\bigr)^2} \, d\bar x \wedge dx
\end{equation}
and conclude from \eqref{eq: wedge dx dxbar} that $A^\lambda$ is anti-self-dual. Writing more explicitly, we have
\begin{equation*}
    A^\lambda(x) = \sum_i A_i(x) dx_i \qquad \text{and} \qquad F^\lambda(x) = \sum_{i<j} F_{ij}(x) dx_i \wedge dx_j
\end{equation*}
with components
\begin{equation*}
    A_i(x) = \frac{1}{\lambda^2+|x|^2} \bigl(\delta_{ik}x_1 - \delta_{i1}x_k + \epsilon_{1ijk}x_j\bigr)\sigma_k \,,
\end{equation*}
\begin{equation*}
    F_{ij}(x) = \frac{2\lambda^2}{\bigl(\lambda^2+|x|^2\bigr)^2}\bigl(\delta_{i1}\delta_{jk} - \delta_{j1}\delta_{ik} - \epsilon_{1ijk}\bigr)\sigma_k \,.
\end{equation*}
From here, we computes
\begin{equation}
    \label{eq: lambda-instanton Z}
    Z_{ij} = g^{pq}\langle F_{ip}, F_{jq} \rangle
    = \frac{3\lambda^4}{a^2} \frac{\bigl(1+|x|^2\bigr)^2}{\bigl(\lambda^2+|x|^2\bigr)^4} \delta_{ij} 
    = 
    \frac{3\lambda^4}{4a^4} \left(\frac{1+|x|^2}{\lambda^2+|x|^2}\right)^4 g_{ij}
\end{equation}
as well as\footnote{C.f. p. 116 of \cite{DonaldsonKronheimer} or p. 46 of \cite{LawsonBook}.}
\begin{equation}
    \label{eq: lambda-instanton |F|2}
    |F^\lambda|^2 = g^{ij}Z_{ij} = \frac{3\lambda^4}{a^4} \left(\frac{1+|x|^2}{\lambda^2+|x|^2}\right)^4 \,.
\end{equation}
Observe first that since the squared-norm of the curvature is a gauge and coordinate independent quantity, the instantons $A^{\lambda_1}$ and $A^{\lambda_2}$ can only be equivalent if $\lambda_1 = \lambda_2$. Second, we see that the Yang--Mills energy density $|F^\lambda|^2$ concentrates at the south pole ($x=0$) as $\lambda \to 0$ and is constant equal to $\tfrac{3}{a^4}$ on $\nS^4_a$ for $\lambda = 1$. In particular,
\begin{equation*}
    \int_{S^4} |F^\lambda|^2 d\mu = \int_{S^4} |F^1|^2 d\mu = 8\pi^2
\end{equation*}
and $A^\lambda$ extends to an instanton on $E \to \nS^4_a$ by Uhlenbeck's theorem \cite{UhlenbeckRemovable}. 
\begin{theorem}
    Let $E \to S^4$ be a $\SU(2)$ vector bundle with $c_2(E) = 1$ over the standard four-sphere. Then, the round metric $g_{\nS^4_a}$ of radius $a > 0$ together with the instanton $A^1$ is a 
    \begin{itemize}
        \item shrinking coupled soliton if $a^2 > \tfrac{1}{8}$,

        \item steady coupled soliton if $a^2 = \tfrac{1}{8}$,

        \item expanding coupled soliton if $a^2 < \tfrac{1}{8}$.
    \end{itemize}
    More generally the rescaled round metrics $g(t) = a^2(t)g_{\nS^4}$ together with the instanton $A^1$ satisfy \eqref{eq: * kappa} if and only if
    \begin{equation*}
        \frac{d}{dt}a^2(t) = -6 + \frac{3\kappa(t)}{4a^2(t)} \,.
    \end{equation*}
    In particular, if $\kappa(t) = \kappa$ is constant, then the radius $a(t)$ converges exponentially fast to $\sqrt{\frac{\kappa}{8}}$ as $t \to \infty$.
\end{theorem}
\begin{proof}
    By inserting $\lambda = 1$ in \eqref{eq: lambda-instanton Z} we see that $A^1$ has the property that $Z$ is proportional to $g_{\nS^4_a}$. In fact, using the conformal invariance of the Yang--Mills equation and that the Ricci curvature of \eqref{eq: round metric stereograph} is
    \begin{equation*}
        \R_{ij} = \frac{12}{\bigl(1+|x|^2\bigr)^2} \delta_{ij} = \frac{3}{a^2} g_{ij} \,,
    \end{equation*}
    we learn that $(g_{\nS^4_a}, A^1)$ satisfies
    \begin{equation*}
        2\R_{ij} - Z_{ij} = \frac{3}{a^2}\left(2 - \frac{1}{4a^2}\right) g_{ij} \,, \qquad D^*F = 0
    \end{equation*}
    and is thus a shrinking coupled soliton with $X\equiv0$ and $\omega\equiv0$. The sign of $\rho$ depends on $a$ as stated. More generally, the coupled flow \eqref{eq: * kappa} on $(a^2(t)g_{\nS^4}, A^1)$ reduces to 
     \begin{equation*}
        \frac{d}{dt}a^2(t) = -6 + \frac{3\kappa(t)}{4a^2(t)} \,.
    \end{equation*}
    If $\kappa(t) = \kappa$ is constant, the vector field on the right-hand side of the ODE has a unique fixed point at $a^2 = \frac{\kappa}{8}$, while it is strictly positive for $a^2 < \tfrac{\kappa}{8}$ and strictly negative for $a^2 > \tfrac{\kappa}{8}$. If we denote $x(t) = a^2(t) - \tfrac{\kappa}{8}$, the ODE can be rewritten as
    \begin{equation*}
        \frac{dx}{dt} = F(x) = -6 + \frac{6\kappa}{8x + \kappa} \,.
    \end{equation*}
    Linearizing the vector field around $x = 0$,
    \begin{equation*}
        F(x) = -48x + \BO(x) \qquad \text{as } \; x \to 0
    \end{equation*}
    and thus we conclude that 
    \begin{equation*}
        \left|a^2(t) - \frac{\kappa}{8}\right| \leq C(\kappa, a(0)) e^{-48t}
    \end{equation*}
    as desired.
\end{proof}

\section{Coupled energy and entropy}
\label{Coupled energy and entropy}
In this section, we construct several functionals that are monotone along the coupled evolution \eqref{eq: * kappa} depending on $\kappa$. They are analogs of Perelman's energy and entropy for Ricci flow. The coupled energy for the case $\kappa \equiv 1$ was first derived by Streets in \cite{StreetsThesis}.

We define the conjugate heat operator
\begin{equation*}
    \square^* \coloneqq -\frac{\partial}{\partial t} - \Delta + \R - \frac{\kappa}{2}|F|^2
\end{equation*}
which is characterized by the property
\begin{equation}
    \label{eq: conjugate op def}
    \frac{d}{dt}\int_M vu \, d\mu = \int_M \bigl(\square v \, u - v \, \square^* u\bigr) \, d\mu \,.
\end{equation}
Here, $\square \coloneq \partial_t - \Delta$ is the standard heat operator. In particular, solutions to the conjugate heat equation $\square^*u = 0$ are mass-preserving in the sense that
\begin{equation*}
    \frac{d}{dt} \int_M u = -\int_M \square^*u \,d\mu = 0 \,.
\end{equation*}

\subsection{Coupled energy}
\label{Coupled energy}
Let $u$ be a smooth positive function on $M$ and consider the functional
\begin{equation}
    \label{eq: coupled energy}
    \mathcal F_\kappa(g,A,u) \coloneq \int_M \left(|\nabla\ln u|^2 + \R - \frac{\kappa}{4}|F|^2\right) u \, d\mu \,.
\end{equation}
Since $M$ is compact, integration by parts shows that $\mathcal F_\kappa(g,A,u) = \int_M f d\mu$ with
\begin{equation}
    \label{eq: f energy}
    f \coloneqq \left(-2\Delta\ln u - |\nabla\ln u|^2 + \R - \frac{\kappa}{4}|F|^2\right) u \,.
\end{equation}
We aim to show that if $u$ is a solution to the conjugate heat equation along \eqref{eq: * kappa}, then $\square^*f \leq 0$. Before delving into the proof, we need the following computational Lemmas.
\begin{lemma}
    \label{lemma: square computation 1}
    \begin{align*}
        \Bigl|\Ric &- \frac{\kappa}{2}Z -\nabla^2\ln u\Bigr|^2 = |\Ric|^2 - \kappa\,\langle\Ric, Z\rangle + \frac{\kappa^2}{4}|Z|^2 - 2\langle\Ric, \nabla^2\ln u\rangle 
        \\
        &+ \kappa\langle Z, \nabla^2\ln u\rangle + \frac{1}{2}\Delta|\nabla\ln u|^2 - \langle \nabla\Delta\ln u, \nabla\ln u\rangle - \Ric(\nabla\ln u, \nabla\ln u)
    \end{align*}
\end{lemma}
\begin{proof}
    The claim follows directly by combining the straightforward computation
    \begin{align*}
        \Bigl|&\Ric - \frac{\kappa}{2}Z - \nabla^2\ln u\Bigr|^2 = \left|\Ric - \frac{\kappa}{2}Z\right|^2 - 2\left\langle\Ric - \frac{\kappa}{2}Z, \nabla^2\ln u\right\rangle + |\nabla^2\ln u|^2
        \\
        &= |\Ric|^2 - \kappa\,\langle\Ric, Z\rangle + \frac{\kappa^2}{4}|Z|^2 - 2\langle\Ric, \nabla^2\ln u\rangle + \kappa\langle Z, \nabla^2\ln u\rangle + |\nabla^2\ln u|^2
    \end{align*}
    and the Bochner formula
    \begin{equation*}
        |\nabla^2\ln u|^2 = \frac{1}{2}\Delta|\nabla\ln u|^2 - \langle \nabla\Delta\ln u, \nabla\ln u\rangle - \Ric(\nabla\ln u, \nabla\ln u) \,.
    \end{equation*}
\end{proof}
\begin{lemma}
    \label{lemma: square computation 2}
    \begin{equation*}
        |D^*F - F(\nabla\ln u, \cdot)|^2 = |D^*F|^2 + 2 \left\langle\Div Z - \frac{1}{4}\nabla|F|^2, \nabla\ln u\right\rangle + Z(\nabla\ln u, \nabla\ln u)
    \end{equation*}
\end{lemma}
\begin{proof}
    A simple expansion of the square gives
    \begin{equation*}
        |D^*F - F(\nabla\ln u, \cdot)|^2 = |D^*F|^2 + 2\bigl\langle\nabla^kF_{kj}, F^{ij}\bigr\rangle \nabla_i\ln u + Z(\nabla\ln u, \nabla\ln u) \,.
    \end{equation*}
    Now using the first divergence identity
    \begin{equation*}
        \nabla^kZ_{ki} = \bigl\langle\nabla^kF_{kj}, F_i^{\;j}\bigr\rangle + \frac{1}{4}\nabla_i|F|^2
    \end{equation*}
    of \cref{lemma: divergence identities}, we can write the second term as
    \begin{equation*}
        2\bigl\langle\nabla^kF_{kj}, F^{ij}\bigr\rangle \nabla_i\ln u = 2 \left\langle\Div Z - \frac{1}{4}\nabla|F|^2, \nabla\ln u\right\rangle
    \end{equation*}
    and the claim follows.
\end{proof}
Now we can proof the following \textit{differential Harnack} type result.
\begin{theorem}
    \label{thm: differential harnack for energy}
    Let $\bigl(g(t), A(t)\bigr)$ be a solution of \eqref{eq: * kappa} and let $u$ be a positive, smooth solution of the conjugate heat equation. Then, the function $f$ defined in \eqref{eq: f energy} satisfies
    \begin{align*}
        \square^*f &= -2\left|\Ric - \frac{\kappa}{2}Z - \nabla^2\ln u\right|^2 u - \kappa\bigl|D^*F - F(\nabla\ln u, \cdot)\bigr|^2u + \frac{\kappa'}{4} |F|^2u \,.
    \end{align*}
    In particular, if $\kappa$ is non-increasing, then $f$ is a subsolution to the conjugate heat equation.
\end{theorem}
\begin{proof}
    Since $\square^*u = 0$, it follows from the Leibniz rule that
    \begin{equation}
        \label{eq: simplif conjug}
        \square^*(vu) = u\left(-\frac{\partial}{\partial t} - \Delta\right)v - 2u \, \bigl\langle\nabla v, \nabla\ln u \bigr\rangle
    \end{equation}
    for any $v \in C^\infty(M)$ and in particular,
    \begin{align*}
        \notag
        \square^*f &= - \left(\frac{\partial}{\partial t} + \Delta\right)\left(-2\Delta\ln u - |\nabla\ln u|^2 + \R - \frac{\kappa}{4}|F|^2\right) u
        \\
        &\quad - 2 \left\langle\nabla\ln u, \nabla\left(-2\Delta\ln u - |\nabla\ln u|^2 + \R - \frac{\kappa}{4}|F|^2\right)\right\rangle u \,.
    \end{align*}
    Using \cite[Proposition 2.3.10]{ToppingBook}, we compute the first term
    \begin{align*}
        \left(\frac{\partial}{\partial t} + \Delta\right)&\Delta\ln u = \Delta^2\ln u + \Delta\left(\frac{\partial}{\partial t}\ln u\right) 
        \\
        &\quad + \bigl\langle2\Ric - \kappa Z, \nabla^2\ln u\bigr\rangle - \kappa \left\langle\Div Z - \frac{1}{2}\nabla|F|^2, \nabla\ln u\right\rangle
        \\
        &= \Delta^2\ln u + \Delta\left(-\Delta\ln u - |\nabla\ln u|^2 + \R - \frac{\kappa}{2}|F|^2\right) 
        \\
        &\quad + 2\bigl\langle\Ric, \nabla^2\ln u\bigr\rangle - \kappa \bigl\langle Z, \nabla^2\ln u\bigr\rangle - \kappa \left\langle\Div Z - \frac{1}{2}\nabla|F|^2, \nabla\ln u\right\rangle
        \\
        &= -\Delta|\nabla\ln u|^2 + \Delta\R - \frac{\kappa}{2} \Delta|F|^2
        \\
        &\quad + 2\bigl\langle\Ric, \nabla^2\ln u\bigr\rangle - \kappa \bigl\langle Z, \nabla^2\ln u\bigr\rangle - \kappa \left\langle\Div Z - \frac{1}{2}\nabla|F|^2, \nabla\ln u\right\rangle \,.
    \end{align*}
    Similarly, the second term yields
    \begin{align*}
        \left(\frac{\partial}{\partial t} + \Delta\right)|\nabla\ln u|^2 &= \Delta|\nabla\ln u|^2 + 2\Ric(\nabla\ln u, \nabla\ln u) - \kappa Z(\nabla\ln u, \nabla\ln u)
        \\
        &\quad + 2\left\langle\nabla\ln u, \nabla\left(-\Delta\ln u - |\nabla\ln u|^2 + \R - \frac{\kappa}{2}|F|^2\right)\right\rangle
        \\
        &= \Delta|\nabla\ln u|^2 + 2\Ric(\nabla\ln u, \nabla\ln u) - \kappa Z(\nabla\ln u, \nabla\ln u)
        \\
        &\quad + 2\left\langle\nabla\ln u, \nabla\left(-2\Delta\ln u - |\nabla\ln u|^2 + \R - \frac{\kappa}{4}|F|^2\right)\right\rangle 
        \\
        &\quad +2 \langle\nabla\ln u, \nabla\Delta\ln u\rangle - \frac{\kappa}{2}\langle\nabla\ln u, \nabla|F|^2\rangle \,.
    \end{align*}
    For the scalar curvature term, the evolution equation from \cref{prop: evol eq Riemannian} simply gives
    \begin{equation*}
        \left(\frac{\partial}{\partial t} + \Delta\right) \R = 2\,\Delta\R + 2\,|\Ric|^2 - \kappa \langle\Ric,Z\rangle - \kappa\,\Delta|F|^2 + \kappa \Div^2S \,,
    \end{equation*}
    whereas the gauge curvature term can be computed using \cref{prop: evol eq bundle 2}
    \begin{align*}
        \left(\frac{\partial}{\partial t} + \Delta\right)\left(\frac{\kappa}{4}|F|^2\right) &= \frac{\kappa'}{4}|F|^2 + \frac{\kappa}{4}\Delta|F|^2 + \frac{\kappa}{4}\frac{\partial}{\partial t}|F|^2
        \\
        &= \frac{\kappa'}{4}|F|^2 + \kappa\langle\Ric, Z\rangle - \frac{\kappa^2}{2}|Z|^2 + \kappa\Div^2S - \kappa|D^*F|^2 \,.
    \end{align*}
    Combining all four the terms, together with \cref{lemma: square computation 1} and \cref{lemma: square computation 2}, we obtain
    \begin{align*}
        \biggl(&\frac{\partial}{\partial t} + \Delta\biggr)\left(-2\Delta\ln u - |\nabla\ln u|^2 + \R - \frac{\kappa}{4}|F|^2\right)
        \\
        &= \Delta|\nabla\ln u|^2 - 4\langle\Ric,\nabla^2\ln u\rangle + 2\kappa \langle Z, \nabla^2\ln u\rangle + 2|\Ric|^2 - 2\kappa\langle\Ric, Z\rangle 
        \\
        &\quad+ 2\kappa \left\langle\Div Z - \frac{1}{2}\nabla|F|^2, \nabla\ln u\right\rangle - 2\Ric(\nabla\ln u, \nabla\ln u) + \kappa Z(\nabla\ln u, \nabla\ln u)
        \\
        &\quad - 2\left\langle\nabla\ln u, \nabla\left(-2\Delta\ln u - |\nabla\ln u|^2 + \R - \frac{\kappa}{4}|F|^2\right)\right\rangle + \frac{\kappa^2}{2}|Z|^2 + \kappa|D^*F|^2
        \\
        &\quad - 2 \langle\nabla\ln u, \nabla\Delta\ln u\rangle + \frac{\kappa}{2}\langle\nabla\ln u, \nabla|F|^2\rangle - \frac{\kappa'}{4}|F|^2
        \\
        &= 2\left|\Ric - \frac{\kappa}{2}Z - \nabla^2\ln u\right|^2 + \kappa\Big|D^*F - F(\nabla\ln u, \cdot)\Big|^2
        \\
        &\quad -2\left\langle\nabla\ln u, \nabla\left(-2\Delta\ln u - |\nabla\ln u|^2 + \R - \frac{\kappa}{4}|F|^2\right)\right\rangle - \frac{\kappa'}{4}|F|^2
    \end{align*}
    and the assertion follows.
\end{proof}
Inserting $v = 1$ and $u = f$ in \eqref{eq: conjugate op def} and applying \cref{thm: differential harnack for energy} we establish
\begin{theorem}
    \label{thm: F monotonicity}
    Let $\bigl(g(t), A(t)\bigr)$ be a solution of \eqref{eq: * kappa} and let $u$ be a positive smooth solution of the conjugate heat equation. Then,
    \begin{align*}
        \frac{d}{dt}\mathcal F_\kappa(g,A,u) &= 2\int_M \left|\Ric - \frac{\kappa}{2}Z - \nabla^2\ln u\right|^2 u \, d\mu 
        \\
        &\quad + \kappa\int_M \Big|D^*F - F(\nabla\ln u, \cdot)\Big|^2u \, d\mu - \frac{\kappa'}{4} \int_M |F|^2 u \, d\mu \,.
    \end{align*}
    In particular, if $\kappa$ positive non-increasing, then $\mathcal F_\kappa$ is non-decreasing along \eqref{eq: * kappa}.
\end{theorem}
Since the gauge curvature term in the definition \eqref{eq: coupled energy} of the coupled energy has a negative sign, we derive as a consequence of \cref{thm: F monotonicity} the following Poincaré-type inequality. We denote by $\lambda_\kappa(t)$ be the first eigenvalue of the Schrödinger-type operator $-4\Delta + \R - \frac{\kappa}{4}|F|^2$ at time $t$.
\begin{proposition}
    \label{prop: energy inequality}
    Let $\bigl(g(t), A(t)\bigr)$ be a solution of \eqref{eq: * kappa} with $\kappa$ positive non-increasing for all $t \in [0,T)$. Then, the estimate
    \begin{equation}
        \label{eq: energy inequality}
        \frac{\kappa}{4}\int_M |F|^2 \varphi^2 \, d\mu + \lambda_\kappa(0)\int_M \varphi^2 \, d\mu \leq 4\int_M |\nabla \varphi|^2 \, d\mu + \int_M \R \varphi^2 \, d\mu
    \end{equation}
    holds for every $\varphi \in W^{1,2}(M)$ and at any time $t \in [0,T)$.
\end{proposition}
\begin{proof}
    Consider the Rayleigh quotient
    \begin{equation*}
        \mathcal Q(v) = \frac{\int_M 4|\nabla v|^2 + \left(\R - \frac{\kappa}{4}|F|^2\right)v^2}{\int_M v^2}
    \end{equation*}
    associated to the operator $-4\Delta + \R - \frac{\kappa}{4}|F|^2$, where all the objects are at the initial time $t = 0$. It is not hard to see using direct variational methods that the infimum of $\mathcal Q$ in $W^{1,2}(M) \setminus \{0\}$ is attained by a smooth positive function $v_0$ satisfying the eigenvalue equation
    \begin{equation*}
        -4\Delta v_0 + \R v_0 - \frac{\kappa}{4}|F|^2 v_0 = \lambda_\kappa(0) v_0 \,.
    \end{equation*}
    Now, fix some $t \in [0,T)$ and $\varphi \in C^{\infty}(M)$, and consider the unique smooth solution $u : M \times [0,t] \to \nR$ of the (backward) linear parabolic problem
    \begin{equation*}
        \begin{dcases}
            \square^*u = 0 \quad &\text{in } \; M \times [0,t)
            \\
            u(t) = \varphi^2 &\text{on } \; M
        \end{dcases}
        \qquad .
    \end{equation*}
    Then, using the monotonicity of \cref{thm: F monotonicity}, the $\kappa' \leq 0$ hypothesis, and the mass-preserving property of $u$, we estimate 
    \begin{align*}
        \lambda_\kappa(0) &= \mathcal Q(v_0) \leq \mathcal Q(u^{1/2}(0)) = \frac{\mathcal F_{\kappa(0)}\bigl(g(0),A(0),u(0)\bigr)}{\int_M u(0)} \leq \frac{\mathcal F_{\kappa(t)}\bigl(g(t),A(t),\varphi^2\bigr)}{\int_M \varphi^2} \,.
    \end{align*}
    Multiplying both sides by $\int_M \varphi^2$ we get
    \begin{equation*}
        \lambda_\kappa(0)\int_M \varphi^2 \leq 4\int_M |\nabla \varphi|^2 + \int_M \R \varphi^2 - \frac{\kappa}{4}\int_M |F|^2 \varphi^2 \,.
    \end{equation*}
    The general statement for $\varphi \in W^{1,2}(M)$ follows by density.
\end{proof}
As a consequence, we establish \cref{Theorem preserved integral inequality}:
\begin{corollary}
    Let $\bigl(g(t), A(t)\bigr)$ be a solution of \eqref{eq: * kappa} with $\kappa$ positive non-increasing for all $t \in [0,T)$. Then, the inequality
    \begin{equation*}
        \frac{\kappa}{4} \int_M |F|^2 \, d\mu \leq \int_M \R \, d\mu
    \end{equation*}
    is preserved in time.
\end{corollary}
\begin{proof}
    Suppose that the inequality holds at time $t = 0$. Then, the operator $-4\Delta + \R - \frac{\kappa}{4}|F|^2$ is non-negative and $\lambda_\kappa(0) \geq 0$. Applying \eqref{eq: energy inequality} with $\varphi \equiv 1$ at any later time $t > 0$ gives
    \begin{equation*}
        \frac{\kappa}{4}\int_M |F|^2 \,d\mu \leq \frac{\kappa}{4}\int_M |F|^2 \,d\mu + \lambda_\kappa|M_t| \leq \int_M \R d\mu \,.
    \end{equation*}
\end{proof}

\subsection{Coupled entropy}
Next, we extend the coupled energy to a scaling invariant version. For a smooth positive function $u$ on $M$ and constants $\tau > 0$ $\alpha \in \nR$, we consider the functional
\begin{align*}
    \mathcal W_\kappa^{\alpha}(g,A,u,\tau) &\coloneqq \tau\int_M \left(\R - \frac{\kappa}{4}|F|^2 + |\nabla\ln u|^2\right) u d\mu \\
    &\qquad+\alpha\int_M \left(\ln u + \frac{n}{2}\ln(4\pi\tau)\right)u d\mu \,.
\end{align*}
Again, by integrating by parts, we can write $\mathcal W_\kappa^\alpha(g,A,u,\tau) = \int_M w d\mu$ with
\begin{equation}
    \label{eq: w emtropy}
    w \coloneqq \tau\left(-2\Delta\ln u - |\nabla\ln u|^2 + \R - \frac{\kappa}{4}|F|^2\right) u - \alpha\left(\ln u + \frac{n}{2}\ln(4\pi\tau)\right)u \,.
\end{equation}
Note that for any $\lambda > 0$, there holds $\mathcal W_\kappa^\alpha\bigl(\lambda^2g,A,u,\tau\bigr) = \mathcal W_{\lambda^{-2}\kappa}^\alpha\bigl(g,A,\lambda^nu,\lambda^{-2} \tau\bigr)$.
\begin{theorem}
    \label{thm: differential harnack}
    Let $\bigl(g(t), A(t)\bigr)$ be a solution of \eqref{eq: * kappa}, $u$ be a positive smooth solution of the conjugate heat equation and define $\tau(t) = t_0 - \alpha t$ for some $t_0 > 0$ and $\alpha \in \nR$. Then, the function $w$ defined in \eqref{eq: w emtropy} satisfies
    \begin{align*}
        \square^*w &= -2\tau\left|\Ric - \frac{\kappa}{2}Z - \nabla^2\ln u - \frac{\alpha g}{2\tau}\right|^2 u - \tau\kappa\bigl|D^*F - F(\nabla\ln u, \cdot)\bigr|^2u
        \\
        &\quad + \frac{\kappa'\tau + \alpha\kappa}{4}|F|^2 u
    \end{align*}
\end{theorem}
\begin{proof}
    In view of the formula \eqref{eq: simplif conjug}, we may compute the second term in $\square^*w$ as
    \begin{align*}
    \square^* \left[\left(\ln u + \frac{n}{2}\ln(4\pi\tau)\right)u\right] &= \left(-\frac{\partial}{\partial t} - \Delta\right)\left(\ln u + \frac{n}{2}\ln(4\pi\tau)\right) u - 2|\nabla\ln u|^2 u
    \\
    &= \left(-|\nabla\ln u|^2 - \R + \frac{\kappa}{2}|F|^2 - \frac{n\tau'}{2\tau}\right) u \,,
    \end{align*}
    where we have used that
    \begin{equation}
        \label{eq: evol ln u}
        \left(\frac{\partial}{\partial t} + \Delta\right)\ln u = - |\nabla\ln u|^2 + \R - \frac{\kappa}{2}|F|^2
    \end{equation}
    due to $\square^*u = 0$. Combined with the computations in the proof of \cref{thm: differential harnack for energy}, we obtain
    \begin{align*}
        \square^*w &= -\tau'\left(-2\Delta\ln u - |\nabla\ln u|^2 + \R - \frac{\kappa}{4}|F|^2\right) u
        \\
        &\quad -2\tau\left|\Ric - \frac{\kappa}{2}Z - \nabla^2\ln u\right|^2 u - \kappa\tau\bigl|D^*F - F(\nabla\ln u, \cdot)\bigr|^2u + \frac{\tau\kappa'}{4} |F|^2u
        \\
        &\quad - \alpha\left(-|\nabla\ln u|^2 - \R + \frac{\kappa}{2}|F|^2 - \frac{n\tau'}{2\tau}\right) u
        \\
        &= -2\tau\left|\Ric - \frac{\kappa}{2}Z - \nabla^2\ln u + \frac{\tau'g}{2\tau}\right|^2 u - \tau\kappa\bigl|D^*F + F(\nabla\ln u, \cdot)\bigr|^2u
        \\
        &\quad -\tau'\left(-|\nabla\ln u|^2 - \R + \frac{3\kappa}{4}|F|^2 - \frac{n\tau'}{2\tau}\right) u + \frac{\tau\kappa'}{4} |F|^2u
        \\
        &\quad - \alpha\left(-|\nabla\ln u|^2 - \R + \frac{\kappa}{2}|F|^2 - \frac{n\tau'}{2\tau}\right) u \,,
    \end{align*}
    where $\tau' = \frac{d\tau}{dt}$. Inserting $\tau' = -\alpha$, the assertion follows.
\end{proof}
As before, using \eqref{eq: conjugate op def} we immediately establish the following.
\begin{theorem}
    \label{thm: W monotonicity}
    Let $\bigl(g(t), A(t)\bigr)$ be a solution of \eqref{eq: * kappa}, $u$ be a positive smooth solution of the conjugate heat equation and define $\tau(t) = \tau_0 - \alpha t$ for some $\tau_0 > 0$ and $\alpha \in \nR$. Then,
    \begin{align*}
        \frac{d}{dt}\mathcal W_\kappa(g,A,u,\tau,\alpha) &= 2\tau\int_M \left|\Ric - \frac{\kappa}{2}Z - \nabla^2\ln u - \frac{\alpha g}{2\tau}\right|^2 u \, d\mu 
        \\
        &\quad + \kappa\tau\int_M \Big|D^*F - F(\nabla\ln u, \cdot)\Big|^2u \, d\mu - \frac{\alpha\kappa + \kappa'\tau}{4} \int_M |F|^2 u \, d\mu \,.
    \end{align*}
\end{theorem}
By setting $\kappa(t) = \tau(t)$, we establish the following monotonicity formula (\cref{theorem W monotonicity} is simply the $\alpha = 1$ case).
\begin{theorem}
    \label{thm: W monotonicity}
    Let $\bigl(g(t), A(t)\bigr)$ be a solution of \eqref{eq: * kappa} with $\kappa(t) = t_0 - \alpha t$ for some $t_0 > 0$ and $\alpha \in \nR$, and let $u$ be a positive smooth solution of the conjugate heat equation. Then, the coupled entropy
    \begin{align*}
        W^\alpha_\kappa(g,A,u) & \coloneqq \kappa\int_M \left(\R - \frac{\kappa}{4}|F|^2 + |\nabla\ln u|^2\right) u d\mu \,-\, \alpha\int_M \left(\ln u + \frac{n}{2}\ln(4\pi\kappa)\right)u d\mu 
    \end{align*}
    is non-decreasing in time as long as $\kappa(t) > 0$ and constant exactly on coupled gradient solitons.
\end{theorem}

\section{Curvature estimates}
\label{A priori estimates I}
The goal of this section is to establish parabolic curvature estimates assuming the Riemannian curvature is uniformly bounded in time.

\subsection{Uniform bounds}
\label{Uniform bounds}
We begin by exploiting the good coupling term in the evolution equation
\begin{equation*}
    \frac{\partial}{\partial t} |F|^2 = \Delta |F|^2 - 2 |\nabla F|^2 - 2\kappa \, |Z|^2 + 2 \R_{ijkl} \langle F^{ij}, F^{kl} \rangle - 4 \langle F_i^{\;p}, [F^{iq}, F_{pq}] \rangle
\end{equation*}
from \cref{prop: evol eq bundle 1} to control the bundle curvature in terms of the Riemann curvature. We remark that the last term vanishes if the structure group of $E$ is abelian and that the third term is non-positive if the Riemannian curvature operator, seen as a symmetric bilinear form acting on two-forms by
\begin{equation*}
    \Rm(\omega, \phi) \coloneqq \R_{ijkl}\omega^{ij}\phi^{kl} \,,
\end{equation*}
is negative semi-definite. We assume neither here. Due to $\tfrac{1}{n}|F|^4 \leq |Z|^2$ and the Cauchy--Schwarz inequality, we have
\begin{equation}
    \label{eq: full ineq F2}
    \left(\frac{\partial}{\partial t} - \Delta\right) |F|^2 \leq - 2 |\nabla F|^2 - \frac{2\kappa}{n} \, |F|^4 + 2|\Rm||F|^2 + 8|F|^3
\end{equation}
and using Young's inequality in the cubic term
\begin{equation*}
    8|F|^3 \leq \frac{16n}{\kappa}|F|^2 + \frac{\kappa}{n}|F|^4 \,,
\end{equation*}
we obtain
\begin{equation}
    \label{eq: ddt F2 with young}
    \left(\frac{\partial}{\partial t} - \Delta\right) |F|^2 \leq - 2 |\nabla F|^2 - \frac{\kappa}{n} \, |F|^4 + 2\left(|\Rm| + \frac{8n}{\kappa}\right)|F|^2 \,.
\end{equation}
It is clear at this point that a uniform upper bound on $|\Rm|$ and a lower bound on $\kappa$ imply a bound on $|F|$. In fact, a uniform $L^q$ bound on $|\Rm|$ with $q > \tfrac{n}{2}$ and a uniform Sobolev inequality suffice in order to get a supremum bound on $|F|$ via Moser's iteration, see \cite[Section 2.6]{RoyoThesis}. We use the good coupling term to do better.
\begin{lemma}
    \label{lemma: max princ}
    Let $u : M \times [0, T) \to \nR$ be a smooth non-negative function satisfying
    \begin{equation*}
        \left(\frac{\partial}{\partial t} - \Delta\right) u \leq - A(t) \, u^2 + B(t) \, u \,,
    \end{equation*}
    where $A(t) > 0$ and $B(t) \geq 0$ are smooth, time-dependent coefficients. Suppose that $A'(t) \leq 0$ and $B'(t) \geq 0$ for all times $t \in [0, T)$. Then,
    \begin{equation*}
        u \leq \frac{1}{At} + \frac{B}{A}
    \end{equation*}
    holds in $M \times [0,T)$. Moreover, if $u \leq B/A$ initially, it remains so for all positive times.
\end{lemma}
\begin{proof}
    For the first part, consider the function $v \coloneq tu$ which satisfies
    \begin{equation*}
        \left(\frac{\partial}{\partial t} - \Delta\right) v \leq \Big(1 - Av + Bt\Big)\frac{v}{t}
    \end{equation*}
    with initial data $v(0, \cdot) = 0$. We want to show that
    \begin{equation*}
        v(x, t) \leq \frac{1 + B(t) \, t}{A(t)}
    \end{equation*}
    for all $(x,t) \in M \times [0, T)$. Fix some arbitrary $\varepsilon > 0$ and define
    \begin{equation*}
        f_\varepsilon(t) \coloneq \frac{1 + B(t)t}{A(t)} + \varepsilon \,.
    \end{equation*}
    Suppose for a contradiction that there is a first event $(x_0, t_0) \in M \times (0,T)$ such that
    \begin{equation*}
        v(x_0, t_0) = \sup_M v(\cdot, t_0) = f_\varepsilon(t_0) \,.
    \end{equation*}
    However, since $f_\varepsilon(0) > 0$, there must hold that
    \begin{equation*}
        \frac{\partial v}{\partial t}(x_0, t_0) \geq f_\varepsilon'(t_0) = -\frac{A'(t_0)}{A^2(t_0)} + \frac{B(t_0)}{A(t_0)} + \frac{\bigl(B'(t_0)A(t_0) - B(t_0)A'(t_0)\bigr)t_0}{A^2(t_0)} \geq 0
    \end{equation*}
    and therefore
    \begin{align*}
        0 \leq \left(\frac{\partial}{\partial t} - \Delta\right) v(x_0, t_0) \leq \Big(1 - A(t_0) v(x_0, t_0) + B(t_0) t_0\Big) \frac{v(x_0, t_0)}{t_0} < 0
    \end{align*}
    gives a contradiction. Consequently, 
    \begin{equation*}
        v < \frac{1 + B(t)t}{A(t)} + \varepsilon
    \end{equation*}
    for all $\varepsilon > 0$ and the claim follows by letting $\varepsilon \to 0$.
    
    If we assume that $u \leq B/A$ initially, we get a contradiction directly at the first event $(x_0,t_0) \in M\times(0,T)$ such that 
    \begin{equation*}
        u(x_0,t_0) = \sup_M u(\cdot, t_0) = \frac{B(t_0)}{A(t_0)} + \varepsilon \,,
    \end{equation*}
    and the claim again follows by letting $\varepsilon \to 0$.
\end{proof}
\begin{theorem}
    \label{thm: zero order control}
    Let $\bigl(g(t), A(t)\bigr)$ be a solution to \eqref{eq: * kappa} satisfying \eqref{eq: kappa condition} and
    \begin{equation*}
        \kappa_0\|F(0)\|_{L^\infty(M)}^2 + \|\Rm(t)\|_{L^\infty(M)} + \frac{8n}{\kappa_0} \leq K(t) 
    \end{equation*}
    for all $t \in [0, T)$, where $K(t)$ is a non-decreasing function on time. Then, there exists a positive constant $C$ depending only on $n$ and $\tfrac{\kappa_1}{\kappa_0}$ such that
    \begin{equation*}
        \kappa_1 |F(x,t)|^2 \leq C K(t)
    \end{equation*}
    holds fo all $(x,t) \in M \times [0, T)$.
\end{theorem}
\begin{proof}
    Using \eqref{eq: ddt F2 with young} and the hypothesis, we see that the function $u \coloneqq |F|^2$ satisfies
    \begin{align*}
        \left(\frac{\partial}{\partial t} - \Delta\right) u \leq - \frac{\kappa_0}{n}u^2 + 2K(t)u \,,
    \end{align*}
    as well as
    \begin{equation*}
        u(\cdot,0) = |F(\cdot,0)|^2 \leq \frac{K(0)}{\kappa_0} \leq \frac{2nK(0)}{\kappa_0} \,.
    \end{equation*}
    Setting $A(t) = \tfrac{\kappa_0}{n}$ and $B(t) = 2K$, \cref{lemma: max princ} implies that $\kappa_0|F|^2 \leq 2nK$ for all $t \geq 0$ and the statement follows after multiplying by $\frac{\kappa_1}{\kappa_0}$.
\end{proof} 
In the following, \cref{thm: zero order control} will be used with $K(t) = K$ constant, although it could be the blow-up rate of $\Rm$, say $K(t) = C(T-t)^{-1}$.

\subsection{Gradient bounds}
\label{Gradient bounds}
Next, we establish uniform estimates on the derivatives of the Riemann and gauge curvature by carefully combining several evolution equations and applying the parabolic maximum principle. First, we estimate the gradient of the gauge curvature:
\begin{theorem}
    \label{thm: kappa gradient F bound}
    Let $\bigl(g(t), A(t)\bigr)$ be a solution to \eqref{eq: * kappa} satisfying \eqref{eq: kappa condition} and
    \begin{equation*}
        \kappa_0\|F(0)\|_{L^\infty(M)}^2 + \|\Rm(t)\|_{L^\infty(M)} + \frac{8n}{\kappa_0} \leq K 
    \end{equation*}
    for $t \in [0, K^{-1}]$. Then, there exists a positive constant $C$ depending only on $n$ and $\tfrac{\kappa_1}{\kappa_0}$ such that
    \begin{equation*}
        \kappa_1|\nabla F|^2 \leq \frac{CK}{t}
    \end{equation*}
    holds on $M \times (0, K^{-1}]$.
\end{theorem}
\begin{proof}
    In the following, we use $C_1, C_2, \ldots$ to denote some non-negative constants depending only on $n$ and $\tfrac{\kappa_1}{\kappa_0}$. From the evolution equation of \cref{cor: evol |nablaF|},
    \begin{align*}
        \biggl(\frac{\partial}{\partial t} - \Delta\biggr)|\nabla F|^2 &\leq -2|\nabla^2F|^2 + C_1|\Rm||\nabla F|^2 + C_1|F||\nabla F|^2 
        \\
        & \qquad \qquad + \kappa_1 C_1|F|^2|\nabla F|^2 + C_1|F||\nabla\Rm||\nabla F|
        \\
        &\leq -2|\nabla^2F|^2 + \kappa^{-1}_1|\nabla\Rm|^2 + C_2\Bigl(|\Rm| + |F| + \kappa_1|F|^2\Bigr)|\nabla F|^2 \,.
    \end{align*}
    In order to control the $|\nabla\Rm|^2$ term, we use \cref{cor: evol |nablaRm|} to estimate
    \begin{align*}
        \biggl(\frac{\partial}{\partial t} &- \Delta\biggr) |\Rm|^2 \leq - 2|\nabla\Rm|^2 + C_3|\Rm|^3 
        \\
        &\hspace{2.7cm} + \kappa_1 C_3|F|^2|\Rm|^2 + \kappa_1 C_3|\nabla F|^2|\Rm| + \kappa_1 C_3|F||\nabla^2F||\Rm|
        \\
        &\leq - 2|\nabla\Rm|^2 + \kappa_1|\nabla^2F|^2 + C_4\Bigl(|\Rm| + \kappa_1|F|^2\Bigr)|\Rm|^2 + \kappa_1 C_4|\Rm||\nabla F|^2
    \end{align*} 
    and combine them into
    \begin{equation}
        \label{eq: ddt k|nabla F|^2 + |Rm|}
        \begin{split}
            \left(\frac{\partial}{\partial t} - \Delta\right)&\Bigl(\kappa_1|\nabla F|^2 + |\Rm|^2\Bigr)
            \\
            &\leq C_5\Bigl(|\Rm| + |F| + \kappa_1|F|^2\Bigr)\Bigl(\kappa_1|\nabla F|^2 + |\Rm|^2\Bigr) \,.
        \end{split}
    \end{equation}
    Using the hypothesis and \cref{thm: zero order control}, we establish
    \begin{align*}
        \left(\frac{\partial}{\partial t} - \Delta\right)\Bigl(\kappa_1|\nabla F|^2 + |\Rm|^2\Bigr) \leq \kappa_1 C_6K|\nabla F|^2 + C_6K^3 \,.
    \end{align*}
    Finally, in view of
    \begin{equation*}
        \left(\frac{\partial}{\partial t} - \Delta\right)|F|^2 \leq - 2|\nabla F|^2 + 2K|F|^2 \,,
    \end{equation*}
    we see that for times $0 \leq t \leq K^{-1}$ the function
    \begin{equation*}
        u \coloneqq t\Big(\kappa_1|\nabla F|^2 + |\Rm|^2\Big) + \kappa_1C_6|F|^2
    \end{equation*}
    satisfies
    \begin{equation*}
        \left(\frac{\partial}{\partial t} - \Delta\right)u \leq C_7K^2 \,.
    \end{equation*}
    The maximum principle implies that
    \begin{align*}
        t\kappa_1|\nabla F|^2 \leq  \kappa_1 C_6 \sup_M|F(0)|^2 + C_7K^2t \leq C_8K \,,
    \end{align*}
    as desired.
\end{proof}
An induction argument and the evolution equation of higher order derivatives of $F$ and $\Rm$ can be now used to further establish the following Theorem. We use $\nabla^m\Rm$ and $\nabla^mF$ to refer to the $m$-th iterated covariant derivative of the Riemannian and bundle curvature, respectively.
\begin{theorem}
    \label{thm: kappa interior estimates}
    Let $\bigl(g(t), A(t)\bigr)$ be a solution to \eqref{eq: * kappa} satisfying \eqref{eq: kappa condition} and
    \begin{equation*}
        \kappa_0\|F(0)\|_{L^\infty(M)}^2 + \|\Rm(t)\|_{L^\infty(M)} + \frac{8n}{\kappa_0} \leq K 
    \end{equation*}
    for $t \in [0, K^{-1}]$. Then, for every $m \in \nN$ there exists a positive constant $C$ depending only on $n$, $m$ and $\tfrac{\kappa_1}{\kappa_0}$ such that
    \begin{equation*}
        \kappa_1|\nabla^{m} F|^2 + |\nabla^{m-1}\Rm|^2 \leq \frac{CK}{t^m}
    \end{equation*}
    holds on $M\times(0, K^{-1}]$.
\end{theorem}
\begin{proof}
    Throughout the proof, we write $C_1, C_2, \ldots$ to denote some non-negative constants depending only on $n$, $m$ and $\tfrac{\kappa_1}{\kappa_0}$. The $m = 1$ case is proven in \cref{thm: kappa gradient F bound}, so let us assume that 
    \begin{equation}
        \label{eq: induction assumption}
        \kappa_1|\nabla^{k} F|^2 + |\nabla^{k-1}\Rm|^2 \leq \frac{C_1K}{t^k}
    \end{equation}
    holds on $M\times[0, K^{-1}]$ for each $k \leq m$ and deduce it for $k=m+1$. Recall from \cref{cor: evol |nablaF|} that
    \begin{align*}
        \frac{\partial}{\partial t}|\nabla^{m+1}F|^2 &\leq \Delta|\nabla^{m+1}F|^2 - 2 |\nabla^{m+2}F|^2 
        \\
        &\quad + C_2 \sum_{i + j = m+1} |\nabla^{i}\Rm||\nabla^{j}F||\nabla^{m+1}F|
        \\
        &\quad + C_2 \sum_{i + j = m+1} |\nabla^{i}F||\nabla^{j}F||\nabla^{m+1}F| 
        \\
        &\quad + \kappa_1 \, C_2 \sum_{i + j + k = m+1} |\nabla^{i}F||\nabla^{j}F||\nabla^{k}F||\nabla^{m+1}F| \,.
    \end{align*}
    Using Young's inequality and \eqref{eq: induction assumption} repeatedly, the first sum can be estimated by
    \begin{align*}
        S_1 &= C_2|\Rm||\nabla^{m+1}F|^2 + C_2|\nabla^{m+1}\Rm||F||\nabla^{m+1}F| 
        \\
        &\quad + C_2|\nabla^{m}\Rm||\nabla F||\nabla^{m+1}F| + C_2\sum_{i+j=m-2}|\nabla^{i+1}\Rm||\nabla^{j+2}F||\nabla^{m+1}F|
        \\
        &\leq C_3K|\nabla^{m+1}F|^2 + \kappa^{-1}_1|\nabla^{m+1}\Rm|^2 + \kappa_1 C_3|F|^2|\nabla^{m+1}F|^2 
        \\
        &\quad + C_3\kappa^{-\frac{1}{2}}_1t^{-\frac{1}{2}}K^{\frac{1}{2}}|\nabla^{m}\Rm||\nabla^{m+1}F| + C_3\kappa^{-\frac{1}{2}}_1K \sum_{i+j=m-2} t^{-\frac{i+j+4}{2}}|\nabla^{m+1}F|
        \\
        &\leq C_4K|\nabla^{m+1}F|^2 + \kappa^{-1}_1|\nabla^{m+1}\Rm|^2 
        \\
        &\quad + \kappa^{-1}_1t^{-1}|\nabla^{m}\Rm|^2 + C_4\kappa^{-\frac{1}{2}}_1t^{-\frac{m+2}{2}} K |\nabla^{m+1}F|
        \\
        &\leq C_5K|\nabla^{m+1}F|^2 + \kappa_1^{-1}|\nabla^{m+1}\Rm|^2 + \kappa_1^{-1}t^{-1}|\nabla^{m}\Rm|^2 + \kappa_1^{-1}t^{-m+2}K \,.
    \end{align*}
    For the second sum, using that $\kappa_1^{-1} \leq K \leq t^{-1}$ by hypothesis,
    \begin{align*}
        S_2 &= 2C_2|F||\nabla^{m+1}F|^2 + C_2\sum_{i+j=m-1}|\nabla^{i+1}F||\nabla^{j+1}F||\nabla^{m+1}F|
        \\
        &\leq C_6 \kappa_1^{-\frac{1}{2}} K^{\frac{1}{2}} |\nabla^{m+1}F|^2 + C_6\kappa_1^{-1}K\sum_{i+j=m-1} t^{-\frac{i+j+2}{2}}|\nabla^{m+1}F|
        \\
        &\leq C_6K|\nabla^{m+1}F|^2 + C_7\kappa_1^{-1}t^{\frac{m+1}{2}} K |\nabla^{m+1}F|
        \\
        &\leq C_8K|\nabla^{m+1}F|^2 + \kappa_1^{-2}t^{-m+1}K
        \\
        &\leq C_8K|\nabla^{m+1}F|^2 + \kappa_1^{-1}t^{-m+2}K \,.
    \end{align*}
    Similarly for the third one, we get
    \begin{align*}
        S_3 &= 3C_2\kappa_1|F|^2|\nabla^{m+1}F|^2 + 3C_2\kappa_1|F| \sum_{i+j=m-1} |\nabla^{i+1}F||\nabla^{j+1}F||\nabla^{m+1}F|
        \\
        &\qquad + C_2\kappa_1\sum_{i+j+k=m-2}|\nabla^{i+1}F||\nabla^{j+1}F||\nabla^{k+1}F||\nabla^{m+1}F|
        \\
        &\leq C_9K|\nabla^{m+1}F|^2 + C_9 \kappa_1^{-\frac{1}{2}} t^{-\frac{m+1}{2}} K^{\frac{3}{2}} |\nabla^{m+1}F|
        \\
        &\leq C_{10}K|\nabla^{m+1}F|^2 + \kappa_1^{-1} t^{-m+1} K^2
        \\
        &\leq C_{10}K|\nabla^{m+1}F|^2 + \kappa_1^{-1} t^{-m+2}K \,.
    \end{align*}
    Bringing it all together we have
    \begin{equation}
        \label{eq: ddt nabla m+1F}
        \begin{split}
            \frac{\partial}{\partial t}|\nabla^{m+1}F|^2 &\leq \Delta|\nabla^{m+1}F|^2 - 2 |\nabla^{m+2}F|^2 + C_{11}K|\nabla^{(m+1)}F|^2
            \\
            &+ \kappa_1^{-1}|\nabla^{m+1}\Rm|^2 + \kappa_1^{-1}t^{-1}|\nabla^{m}\Rm|^2 + \kappa_1^{-1} t^{-m+2}K \,.
        \end{split}
    \end{equation}
    For derivatives of the Riemann tensor, we have from \cref{cor: evol |nablaRm|} that
    \begin{align*}
        \frac{\partial}{\partial t} |\nabla^{m}\Rm|^2 &\leq \Delta|\nabla^{m}\Rm|^2 - 2 |\nabla^{m + 1}\Rm|^2 
        \\
        &\quad + C_{12}\sum_{i+j=m} |\nabla^{i}\Rm||\nabla^{j}\Rm||\nabla^{m}\Rm|
        \\
        &\quad + \kappa C_{12}\sum_{i+j+k=m} |\nabla^{i}\Rm||\nabla^{j}F||\nabla^{k}F||\nabla^{m}\Rm| 
        \\
        &\quad + \kappa C_{12}\sum_{i+j=m+2} |\nabla^{i}F ||\nabla^{j}F||\nabla^{m}\Rm| \,.
    \end{align*}
    Using the induction hypothesis we estimate the first sum by
    \begin{align*}
        S_1 &= 2C_{12}|\Rm||\nabla^{m}\Rm|^2 + C_{12}\sum_{i+j=m-2}|\nabla^{i+1}\Rm||\nabla^{j+1}\Rm||\nabla^{m}\Rm|
        \\
        &\leq 2C_{12}K|\nabla^{m}\Rm|^2 + C_{13} t^{-\frac{m+2}{2}} K|\nabla^{m}\Rm|
        \\
        &\leq C_{14}K|\nabla^{m}\Rm|^2 + t^{-(m+2)}K \,, 
    \end{align*}
    the second sum by
    \begin{align*}
        S_2 &= C_{12}\kappa_1 |F|^2|\nabla^{m}\Rm|^2 + \kappa_1 C_{12}\sum_{i+j+k=m-1} |\nabla^{i}\Rm||\nabla^{j+1}F||\nabla^{k}F||\nabla^{m}\Rm|
        \\
        &\leq C_{15}K|\nabla^{m}\Rm|^2 +  C_{16} t^{-\frac{m+1}{2}} K^{\frac{3}{2}} |\nabla^{m}\Rm|
        \\
        &\leq C_{17} K|\nabla^{m}\Rm|^2 + t^{-(m+1)} K^{2}
        \\
        &\leq C_{17} K|\nabla^{m}\Rm|^2 + t^{-(m+2)} K \,,
    \end{align*}
    and the third sum by
    \begin{align*}
        S_3 &= 2C_{12}\kappa_1|F||\nabla^{m+2}F||\nabla^{m}\Rm| + 2C_{12}\kappa_1|\nabla F||\nabla^{m+1}F||\nabla^{m}\Rm| 
        \\
        &\qquad + \kappa_1 C_{12}\sum_{i+j=m-2}| \nabla^{i+2}F ||\nabla^{j+2}F||\nabla^{m}\Rm|
        \\
        &\leq \kappa_1|\nabla^{m+2}F|^2 + C_{18}\kappa_1|F|^2|\nabla^{m}\Rm|^2 + C_{18} \kappa_1^{\frac{1}{2}}t^{-\frac{1}{2}} K^{\frac{1}{2}} |\nabla^{m+1}F||\nabla^{m}\Rm|
        \\
        &\qquad + C_{18} t^{-\frac{m+2}{2}} K |\nabla^{m}\Rm|
        \\
        &\leq \kappa_1|\nabla^{m+2}F|^2 + C_{19}K|\nabla^{m}\Rm|^2 + \kappa_1 t^{-1}|\nabla^{m+1}F|^2 + Kt^{-(m+2)} \,.
    \end{align*}
    Writing it all together, we get
    \begin{equation}
        \label{eq: ddt nabla m+1Rm}
        \begin{split}
            \frac{\partial}{\partial t} |\nabla^{m}\Rm|^2 &\leq \Delta|\nabla^{m}\Rm|^2 - 2 |\nabla^{m + 1}\Rm|^2 + C_{20}K|\nabla^{m}\Rm|^2
            \\
            &\qquad + \kappa_1|\nabla^{m+2}F|^2 + \kappa_1 t^{-1}|\nabla^{m+1}F|^2 + Kt^{-(m+2)} \,,
        \end{split}
    \end{equation}
    which in combination with \eqref{eq: ddt nabla m+1F} and the hypothesis $t \leq K^{-1}$, yields
    \begin{align*}
        \left(\frac{\partial}{\partial t} - \Delta\right) \Big(\kappa_1|\nabla^{m+1}F|^2 &+ |\nabla^{m}\Rm|^2\Big) 
        \\
        &\leq C_{21}\kappa_1 t^{-1}|\nabla^{m+1}F|^2 + C_{21}t^{-1}|\nabla^{m}\Rm|^2 + Kt^{-(m+2)} \,.
    \end{align*}
    The missing terms are lower order, so all is left to do is to add enough lower order gradient terms. From \eqref{eq: ddt nabla m+1F} and \eqref{eq: ddt nabla m+1Rm} applied to lower order derivatives, combined with the induction hypothesis, it is easy to see that
    \begin{align*}
        \left(\frac{\partial}{\partial t} - \Delta\right)\Big(\kappa_1|\nabla^{m}F|^2\Big) \leq -2\kappa_1|\nabla^{m+1}F|^2 + |\nabla^{m}\Rm|^2 + C_{22}Kt^{-(m+1)}
    \end{align*}
    and
    \begin{align*}
        \left(\frac{\partial}{\partial t} - \Delta\right)\Big(|\nabla^{m-1}\Rm|^2\Big) \leq -2 |\nabla^{m}\Rm|^2 + \kappa_1 |\nabla^{m+1}F|^2 + C_{22}Kt^{-(m+1)} \,.
    \end{align*}
    Therefore, the function
    \begin{align*}
        u &\coloneqq t^{m+2}\Big(\kappa_1|\nabla^{m+1}F|^2 + |\nabla^{m}\Rm|^2\Big) 
        \\
        &\quad + t^{m+1}\bigl(m+2+C_{21}\bigr)\Big(\kappa_1|\nabla^{m}F|^2 + |\nabla^{m-1}\Rm|^2\Big)
    \end{align*}
    satisfies 
    \begin{align*}
        \left(\frac{\partial}{\partial t} - \Delta\right)u \leq C_{23}K
    \end{align*}
    and it follows from the maximum principle that 
    \begin{align*}
        t^{m+2}\Big(\kappa|\nabla^{m+1}F|^2 + |\nabla^m\Rm|^2\Big) \leq \sup_{M\times\{0\}} u + C_{23}Kt = C_{24}Kt \,.
    \end{align*}
    This concludes the induction step and the proof.
\end{proof}
To finish the section, we estimate the maximal time of existence from below purely in terms of initial data. Moreover, we obtain a version of \cref{thm: kappa interior estimates} where the upper bound $K$ depends only on initial data. First, we need the following Lemma. 
\begin{lemma}
    \label{lemma: ODI max curvatures}
    Let $\bigl(g(t), A(t)\bigr)$ be a solution of \eqref{eq: * kappa} satisfying \eqref{eq: kappa condition} for all $t \in [0, T)$. Then, there exists a positive constant $C$ depending only on $n$ such that the function
    \begin{equation*}
        u(t) \coloneqq \sup_M \Bigl( |\Rm(t)|^2 + \kappa_1|\nabla F(t)|^2 + \kappa^2_1|F(t)|^4 + |F(t)|^2 \Bigr)
    \end{equation*}
    satisfies
    \begin{equation*}
        u(t) \leq \frac{u(0)}{\bigl(1 - C u^{1/2}(0)t\bigr)^2}
    \end{equation*}
    for all $t \in [0,T)$.
\end{lemma}
\begin{proof}
    We write $C_1, C_2, \ldots$ for some non-negative constants depending only on $n$. Let us recall the formulae
    \begin{align*}
        \left(\frac{\partial}{\partial t} - \Delta\right)&\Bigl(\kappa_1|\nabla F|^2 + |\Rm|^2\Bigr) 
        \\
        &\leq C_1\Bigl(|\Rm| + |F| + \kappa_1|F|^2\Bigr)\Bigl(\kappa_1|\nabla F|^2 + |\Rm|^2\Bigr) \,.
    \end{align*}
    and
    \begin{equation*}
        \left(\frac{\partial}{\partial t} - \Delta\right)|F|^2 \leq - 2|\nabla F|^2 + 2\Bigl(|\Rm| + 4|F|\Bigr)|F|^2
    \end{equation*}
    from \eqref{eq: ddt k|nabla F|^2 + |Rm|} and \eqref{eq: ddt F2 with young} respectively. Observe that $C_1$ indeed only depends on $n$. Similarly, one can easily compute
    \begin{equation*}
        \left(\frac{\partial}{\partial t} - \Delta\right) |F|^4 \leq - 2 |\nabla|F|^2|^2 - 2|F|^2|\nabla F|^2 + 4\Bigl(|\Rm| + 4|F|\Bigr)|F|^4 \,. 
    \end{equation*}
    Combining them, we have
    \begin{align*}
        \left(\frac{\partial}{\partial t} - \Delta\right)&\Bigl(|\Rm|^2 + \kappa_1|\nabla F|^2 + \kappa^2_1|F|^4 + |F|^2\Bigr) 
        \\
        &\leq C_2\Bigl(|\Rm| + |F| + \kappa_1|F|^2\Bigr)\Bigl(\kappa_1|\nabla F|^2 + |\Rm|^2\Bigr)
        \\
        &\quad + C_2\Bigl(|\Rm| + |F|\Bigr)\Bigl(\kappa^2_1|F|^4 + |F|^2\Bigr)
        \\
        &\leq C_3 \Bigl(|\Rm|^2 + \kappa_1|\nabla F|^2 + \kappa^2_1|F|^4 + |F|^2\Bigr)^{3/2} \,,
    \end{align*}
    where we have used the simple inequality
    \begin{align*}
        (a + d + c)&(b^2 + a^2) + (a + d)(c^2 + d^2)
        \\
        &\leq \sqrt{3}(a^2 + d^2 + c^2)^{1/2}(b^2 + a^2) + \sqrt{2}(a^2 + d^2)^{1/2}(c^2 + d^2)
        \\
        &\leq (\sqrt{3} + \sqrt{2})(a^2 + b^2 + c^2 + d^2)^{3/2}
    \end{align*}
    for $a,b,c,d \geq 0$ in the last step. We conclude that the function
    \begin{equation*}
        u \coloneqq |\Rm|^2 + \kappa_1|\nabla F|^2 + \kappa^2_1|F|^4 + |F|^2
    \end{equation*}
    satisfies the differential inequality
    \begin{equation*}
        \left(\frac{\partial}{\partial t} - \Delta \right)u \leq C_3 u^{3/2}
    \end{equation*}
    and the claim follows by ODE comparison.
\end{proof}
Finally, we establish \cref{Theorem Short Time}, which we restate for convenience.
\begin{theorem}
    Let $(g, A)$ be arbitrary smooth initial data, $\kappa(t)$ a smooth positive coupling parameter and let
    \begin{equation*}
        K \coloneq \|\Rm\|_{L^\infty(M)} + \|F\|_{L^\infty(M)} + \kappa(0)\|F\|_{L^{\infty}(M)}^2 + \kappa^{1/2}(0)\|\nabla F\|_{L^{\infty}(M)}^2 \,.
    \end{equation*}
    Then, there exist positive constants $\tau=\tau(n)$, $c = c(n,\kappa(0))$ and $C(n,m,\kappa(0))$ such that a smooth solution to \eqref{eq: * kappa} exists on a time interval $[0, \tau K^{-1}]$ and the estimate
    \begin{equation*}
        \|\nabla^{m-1}\Rm(t)\|_{L^\infty(M)}^2 + \kappa(0)\|\nabla^{m}F(t)\|_{L^\infty(M)}^2 \leq \frac{CK}{t^{m}}
    \end{equation*}
    holds for all imes $t \in \bigl(0, cK^{-1}\bigr]$ and positive integer $m$. 
\end{theorem}
\begin{proof}
    Consider the unique solution $(g(t),A(t))$ to \eqref{eq: * kappa} starting at $(g,A)$ on a time interval sufficiently small such that $2^{-1}\kappa(0) \leq \kappa(t) \leq 2\kappa(0)$. Then, \cref{lemma: ODI max curvatures} implies that in such time interval, the function
    \begin{equation*}
        u(t) \coloneqq \sup_M \Bigl( |\Rm(t)|^2 + \kappa(0)|\nabla F(t)|^2 + \kappa^2(0)|F(t)|^4 + |F(t)|^2 \Bigr)
    \end{equation*}
    is bounded by
    \begin{equation*}
        u(t) \leq \frac{u(0)}{\bigl(1 - C(n)u^{1/2}(0)t\bigr)^2} \,.
    \end{equation*}
    In particular, $u(t) \leq 4 u(0)$ for all $t \leq \tfrac{2}{C(n)u^{1/2}(0)} \leq \tau(n) K^{-1}$, and we may bound
    \begin{equation*}
        2^{-1}\kappa(0) \|F(0)\|_{L^\infty(M)}^2 + \|\Rm(t)\|_{L^\infty(M)} + \frac{16n}{\kappa(0)} \leq c(n,\kappa(0)) K 
    \end{equation*}
    for all $t \leq \tfrac{1}{c(n,\kappa(0))K}$. Then, the estimate follows from \cref{thm: kappa interior estimates}.
\end{proof}

\section{Finite-time singularities}
\label{Finite-time singularities}
An important consequence of the curvature estimates in \cref{A priori estimates I} is that the coupled flow \eqref{eq: * kappa} can be extended smoothly as long as the Riemannian curvature remains bounded. 
\begin{theorem}
    \label{thm: T finite implies blow-up}
    Let $\bigl(g(t), A(t)\bigr)$ be a smooth solution of \eqref{eq: * kappa} satisfying $\eqref{eq: kappa condition}$ on a maximal time interval $t \in [0, T)$. If $T < \infty$, then
    \begin{equation*}
        \limsup_{t \to T}\sup_M|\Rm(t)| = \infty \,.
    \end{equation*}
\end{theorem}
\begin{proof}
    Suppose for a contradiction that
    \begin{equation*}
        \sup_M|\Rm(t)| \leq Q
    \end{equation*}
    for all $t \in [0,T)$. Then, there is certainly a $K = K\bigl(n,\kappa_0,A(0),Q\bigr)$ such that
    \begin{equation*}
        \kappa_0\|F(0)\|_{L^\infty(M)}^2 + \|\Rm(t)\|_{L^\infty(M)} + \frac{8n}{\kappa_0} \leq K 
    \end{equation*}
    for all $t \in [0,T)$ and \cref{thm: kappa interior estimates} implies that
    \begin{align*}
        \sup_{t\in[0,T)}\|\nabla^{(m)}\Rm(t)\|_{L^\infty(M)} &< \infty
        \\
        \sup_{t\in[0,T)}\|\nabla^{(m)} F(t)\|_{L^\infty(M)} &< \infty
    \end{align*}
    holds for all $m \in \nN_0$. In turn, this implies that the speed of the flow is bounded uniformly in time,
    \begin{align*}
        \sup_{t\in[0,T)}\|\partial_tg(t)\|_{C^m(M)} &< \infty
        \\
        \sup_{t\in[0,T)}\|\partial_tA(t)\|_{C^m(M)} &< \infty \quad ,
    \end{align*}
    and $g(t)$ and $A(t)$ converge as $t \to T$ to some smooth symmetric tensor field $g(T)$ and $A(T) \in \Omega^1(\mathfrak g_E)$. Moreover, since the speed is uniformly bounded, all metrics are uniformly equivalent and thus $g(T)$ is positive-definite (see \cite[Lemma A.1 \& Proposition A.5]{BrendleBook}). We conclude that $g(T)$ and $A(T)$ define a smooth Riemannian metric and connection, and by short-time existence, we can extend the flow smoothly beyond $T$. This contradicts the maximality of $T$.
\end{proof}
Furthermore, the minimum blow-up rate of the Riemannian curvature at finite-time singularities is that of Ricci flow:
\begin{theorem}
    \label{thm: T finte typ I blowup}
    Let $\bigl(g(t), A(t)\bigr)$ be a smooth solution of \eqref{eq: * kappa} satisfying \eqref{eq: kappa condition} on a maximal time interval $t \in [0, T)$. If $T < \infty$, then there is a positive constant $C$ such that
    \begin{equation*}
        \sup_M|\Rm(t)| \geq \frac{C}{T-t}
    \end{equation*}
    for all $t \in [0,T)$.
\end{theorem}
\begin{proof}
    Consider the function
    \begin{equation*}
        u \coloneqq |\Rm|^2 + \kappa|\nabla F|^2 + \kappa^2|F|^4 + |F|^2 \,.
    \end{equation*}
    Since $T < \infty$, \cref{thm: T finite implies blow-up} implies that the maximum $U(t) \coloneqq \max\{u(t,x) \;:\; x \in M\}$ is positive and $\lim_{t \to T} U(t) = \infty$. Moreover, by \cref{lemma: ODI max curvatures}, it satisfies
    \begin{equation*}
        \frac{dU}{dt} \leq C U^{3/2}
    \end{equation*}
    and consequently
    \begin{equation*}
        \frac{d}{dt}U^{-1} \geq -CU^{-1/2}
    \end{equation*}
    in the sense of forward different quotients, see \cite[§3]{Hamilton86}. Upon integration, we conclude that
    \begin{equation*}
        U(t) \geq \frac{C}{(T-t)^2}
    \end{equation*}
    and thereby
    \begin{equation*}
        \sup_M \Bigl(|\Rm(t)| + |F(t)| + \kappa_1|F(t)|^2 + \kappa_1^{1/2}|\nabla F(t)|\Bigr) \geq \frac{C}{T-t} \,.
    \end{equation*}
    Notice that since $\limsup_{t \to T}\sup_M|\Rm(t)| = \infty$, \cref{thm: zero order control} implies that 
    \begin{equation}
        \label{eq: |F|^2 leq |Rm|}
        \kappa_1 \sup_M |F(t)|^2 \leq C \sup_M|\Rm(t)|
    \end{equation}
    for all $t \in [0,T)$ sufficiently large. Then,
    \begin{equation}
        \label{eq: lowerbound combined}
        \sup_M \Bigl(|\Rm(t)| + \kappa_1^{1/2}|\nabla F(t)|\Bigr) \geq \frac{C}{T-t}
    \end{equation}
    for sufficiently large $t \in [0,T)$. Now, suppose for a contradiction that $\sup_M|\Rm(t)| = o((T-t)^{-1})$, that is, for every $\delta > 0$ there exists some $t_0 \in [0,T)$ such that
    \begin{equation*}
        \sup_M |\Rm(t)| < \frac{\delta}{T-t}
    \end{equation*}
    for all $[t_0,T)$. Then, in view of \eqref{eq: |F|^2 leq |Rm|} and \eqref{eq: ddt k|nabla F|^2 + |Rm|},
    \begin{align*}
        \left(\frac{\partial}{\partial t} - \Delta\right)&\Bigl(\kappa_1|\nabla F|^2 + |\Rm|^2\Bigr) \leq \frac{\delta C}{T-t}\Bigl(\kappa_1|\nabla F|^2 + |\Rm|^2\Bigr)
    \end{align*}
    holds on $M \times [t_0,T)$ for some constant $C > 0$ independent of $\delta$. We deduce by ODE comparison that
    \begin{equation*}
         \kappa_1\sup_M|\nabla F(t)|^2 \leq \sup_M\Bigl(\kappa_1|\nabla F(t_0)|^2 + |\Rm(t_0)|^2\Bigr) \left(\frac{T-t_0}{T-t}\right)^{\delta C}
    \end{equation*}
    for all $t \in [t_0, T)$. In particular, by choosing $\delta = C^{-1}$ at the price of making $t_0 < T$ larger, we conclude that
    \begin{equation*}
        \kappa_1\sup_M|\nabla F(t)|^2 = O\bigl((T-t)^{-1/2}\bigr) \quad \text{as} \; t \to T
    \end{equation*}
    and therefore
    \begin{equation*}
        \sup_M \Bigl(|\Rm(t)| + \kappa_1^{1/2}|\nabla F(t)|\Bigr) = o((T-t)^{-1}) \quad \text{as} \; t \to T \,.
    \end{equation*}
    However, this contradicts \eqref{eq: lowerbound combined} and proves the assertion.
\end{proof}
\cref{Theorem A} follows directly from \cref{thm: T finite implies blow-up} and \cref{thm: T finte typ I blowup}.

\begin{appendix}

\vspace{1cm}
\section{Evolution of curvature derivatives}
\label{Evolution of curvature derivatives}
\subsection{Bundle Curvature}
We denote by $\nabla^{(k)}F$ the $k$-th covariant derivative of the bundle curvature. Let us begin by showing 
\begin{lemma}
    \label{lemma: commutator space}
    For every $m \in \nN$ we have the commutation identity
    \[[\nabla^{(m)}, \Delta]F = \sum_{i + j = m} \nabla^{(i)}\Rm*\nabla^{(j)}F + \sum_{i + j = m}\nabla^{(i)}F*\nabla^{(j)}F \,.
    \]
\end{lemma}
\begin{proof}
    The Ricci identity gives us
    \[[\nabla_i\nabla_j, \nabla_j\nabla_i] F = \Rm*F + F*F \,,
    \]
    from which we obtain
    \begin{align*}
        \nabla\Delta F &= \nabla_i\nabla^p\nabla_pF = \nabla^p\nabla_i\nabla_p F + \Rm*\nabla F + F*\nabla F
        \\
        &= \Delta \nabla F + \nabla(\Rm*F + F*F) + \Rm*\nabla F + F*\nabla F
        \\
        &= \Delta \nabla F + \nabla\Rm*F + \Rm*\nabla F + F*\nabla F \,.
    \end{align*}
    This establishes the case $m = 1$. Now asume that
    \[\nabla^{(m-1)}\Delta F = \Delta \nabla^{(m-1)}F + \sum_{i + j = k-1} \nabla^{(i)}\Rm*\nabla^{(j)}F + \sum_{i + j = m-1}\nabla^{(i)}F*\nabla^{(j)}F \,.
    \]
    Then,
    \begin{align*}
        \nabla^{(m)}\Delta F &= \nabla\nabla^{(m-1)}\Delta F 
        \\
        &= \nabla\left(\Delta \nabla^{(m-1)}F + \sum_{i + j = m-1} \nabla^{(i)}\Rm*\nabla^{(j)}F + \sum_{i + j = m-1}\nabla^{(i)}F*\nabla^{(j)}F\right)
        \\
        &= \Delta \nabla\nabla^{(m-1)}F 
        \\
        &\quad + \nabla\Rm*\nabla^{(m-1)}F + \nabla F*\nabla^{(m-1)}F + \Rm*\nabla^{(m)}F + F*\nabla^{(m)}F
        \\
        &\quad + \sum_{i + j = m-1} \Big(\nabla^{(i+1)}\Rm*\nabla^{(j)}F + \nabla^{(i)}\Rm*\nabla^{(j+1)}F\Big)
        \\
        &\quad + \sum_{i + j = m-1} \Big(\nabla^{(i+1)}F*\nabla^{(j)}F + \nabla^{(i)}F*\nabla^{(j+1)}F\Big)
        \\
        &= \Delta \nabla^{(m)}F + \sum_{i + j = m} \nabla^{(i)}\Rm*\nabla^{(j)}F + \sum_{i + j = m}\nabla^{(i)}F*\nabla^{(j)}F \,,
    \end{align*}
    and the claim follows by induction.
\end{proof}
Next, we need the commutation relation between spatial and time derivatives. Recall that the evolution of the Christoffel symbols and the connection one form $A$ involve terms like
\[\frac{\partial}{\partial t} \Gamma = \nabla \Rm + \kappa \, \nabla F * F \,, \qquad \frac{\partial}{\partial t}A = \nabla F \,.
\]  
Using these, we get
\begin{lemma}
    \label{lemma: commutator spacetime}
    For every $m \in \nN$ we have the commutation identity
    \begin{align*}
        [\partial_t, \nabla^{(m)}]F = \sum_{i + j + 1 = m} \nabla^{(i + 1)}\Rm*\nabla^{(j)}F 
        &+
        \sum_{i + j = m} \nabla^{(i)}F*\nabla^{(j)}F
        \\
        &+ \kappa \sum_{i + j + k = m} \nabla^{(i)}F*\nabla^{(j)}F*\nabla^{(k)}F \,.
    \end{align*}
\end{lemma}
\begin{proof}
    Again we proceed by induction. For $m = 1$,
    \begin{align*}
        \partial_t \nabla F &= \nabla \partial_t F + \partial_t \Gamma * F + \partial_t A * F \\
        &= \nabla \partial_t F + \nabla \Rm * F + \kappa \, \nabla F * F * F + \nabla F * F \,.
    \end{align*}
    Next, assuming that
    \begin{align*}
        \partial_t\nabla^{(m-1)}F &= \nabla^{(m-1)}\partial_tF + \sum_{i + j + 2 = m} \nabla^{(i + 1)}\Rm*\nabla^{(j)}F 
        \\
        &\quad + \sum_{i + j + 1 = m} \nabla^{(i)}F*\nabla^{(j)}F + \kappa \sum_{i + j + k + 1 = m} \nabla^{(i)}F*\nabla^{(j)}F*\nabla^{(k)}F
    \end{align*}
    holds, we compute
    \begingroup
        \allowdisplaybreaks
        \begin{align*}
            \partial_t\nabla^{(m)}F &= \nabla\partial_t\nabla^{(m-1)}F + \partial_t\Gamma*\nabla^{(m-1)}F + \partial_tA*\nabla^{(m-1)}F
            \\
            &= \nabla^{(m)}\partial_tF + \nabla\Rm*\nabla^{(m-1)}F + \kappa \nabla F*\nabla^{(m-1)}F*F + \nabla F*\nabla^{(m-1)}F
            \\
            &\quad + \nabla\Biggl(\sum_{i + j + 2 = m} \nabla^{(i + 1)}\Rm*\nabla^{(j)}F + \sum_{i + j + 1 = m} \nabla^{(i)}F*\nabla^{(j)}F
            \\
            &\hspace{6cm} + \kappa \sum_{i + j + k + 1 = m} \nabla^{(i)}F*\nabla^{(j)}F*\nabla^{(k)}F\Biggr)
            \\
            &= \nabla^{(m)}\partial_tF + \sum_{i + j + 2 = m} \Big(\nabla^{(i + 2)}\Rm*\nabla^{(j)}F + \nabla^{(i + 1)}\Rm*\nabla^{(j + 1)}F\Big)
            \\
            &\hspace{2cm} + \sum_{i + j + 1 = m} \Big(\nabla^{(i + 1)}F*\nabla^{(j)}F + \nabla^{(i)}F*\nabla^{(j + 1)}F\Big)
            \\
            &\hspace{2cm} + \kappa\sum_{i + j + k + 1 = m} \Big(\nabla^{(i + 1)}F*\nabla^{(j)}F*\nabla^{(k)}F + \nabla^{(i)}F*\nabla^{(j + 1)}F*\nabla^{(k)}F 
            \\
            &\hspace{7cm}+ \nabla^{(i)}F*\nabla^{(j)}F*\nabla^{(k + 1)}F\Big)
            \\
            &= \nabla^{(m)}\partial_tF + \sum_{i + j + 1 = m} \nabla^{(i + 1)}\Rm*\nabla^{(j)}F +
            \sum_{i + j = m} \nabla^{(i)}F*\nabla^{(j)}F
            \\
            &\hspace{2cm} + \kappa \sum_{i + j + k = m} \nabla^{(i)}F*\nabla^{(j)}F*\nabla^{(k)}F \,.
    \end{align*}
    \endgroup
\end{proof}
Now we can put \cref{lemma: commutator space} and \cref{lemma: commutator spacetime} together in order compute the evolution of the covariant derivatives of $F$ in complete generality. Using that
\[\frac{\partial}{\partial t} F = \Delta F + F * F + F *\Rm \,,
\]
we obtain
\begin{proposition}
    The $m$-th covariant derivative of the bundle curvature satisfies
    \begin{align*}
        \frac{\partial}{\partial t}\nabla^{(m)} F &= \Delta\nabla^{(m)}F + \sum_{i + j = m} \nabla^{(i)}\Rm*\nabla^{(j)}F +
        \sum_{i + j = m} \nabla^{(i)}F*\nabla^{(j)}F
        \\
        &\hspace{2cm} + \kappa \sum_{i + j + k = m} \nabla^{(i)}F*\nabla^{(j)}F*\nabla^{(k)}F \,.
    \end{align*}
\end{proposition}
\begin{proof}
    \begin{align*}
        \frac{\partial}{\partial t}\nabla^{(m)} F &= \nabla^{(m)}\partial_tF + \sum_{i + j + 1 = k} \nabla^{(i + 1)}\Rm*\nabla^{(j)}F +
        \sum_{i + j = m} \nabla^{(i)}F*\nabla^{(j)}F
        \\
        &\hspace{2cm} + \kappa \sum_{i + j + k = m} \nabla^{(i)}F*\nabla^{(j)}F*\nabla^{(k)}F
        \\
        &= \nabla^{(m)}\Delta F + \sum_{i + j = m} \nabla^{(i)}\Rm*\nabla^{(j)}F +
        \sum_{i + j = m} \nabla^{(i)}F*\nabla^{(j)}F
        \\
        &\hspace{2cm} + \kappa \sum_{i + j + k = m} \nabla^{(i)}F*\nabla^{(j)}F*\nabla^{(k)}F
        \\
        &= \Delta\nabla^{(m)}F + \sum_{i + j = m} \nabla^{(i)}\Rm*\nabla^{(j)}F +
        \sum_{i + j = m} \nabla^{(i)}F*\nabla^{(j)}F
        \\
        &\hspace{2cm} + \kappa \sum_{i + j + k = m} \nabla^{(i)}F*\nabla^{(j)}F*\nabla^{(k)}F \,.
    \end{align*}
\end{proof}
\begin{corollary}
    \label{cor: evol |nablaF|}
    For every $m \in \nN$ we have the evolution equation
    \begin{align*}
        \frac{\partial}{\partial t} |\nabla^{(m)}F|^2 &\leq \Delta|\nabla^{(m)}F|^2 - 2 |\nabla^{(m + 1)}F|^2 + \sum_{i + j = m} \nabla^{(i)}\Rm*\nabla^{(j)}F*\nabla^{(m)}F 
        \\
        &\quad +
        \sum_{i + j = m} \nabla^{(i)}F*\nabla^{(j)}F*\nabla^{(m)}F + \kappa \sum_{i + j + k = m} \nabla^{(i)}F*\nabla^{(j)}F*\nabla^{(k)}F*\nabla^{(m)}F
    \end{align*}
\end{corollary}
\begin{proof}
    Using the evolution equation of the metric and the previous Proposition,
    \begin{align*}
        \frac{\partial}{\partial t} |\nabla^{(m)}F|^2 &= \Rm*\nabla^{(m)}F*\nabla^{(m)}F + \kappa \, F * F*\nabla^{(m)}F*\nabla^{(m)}F + 2 \langle \nabla^{(m)}F, \partial_t\nabla^{(m)}F \rangle
        \\
        &= 2 \langle \nabla^{(m)}F, \Delta\nabla^{(m)}F \rangle + \sum_{i + j = m} \nabla^{(i)}\Rm*\nabla^{(j)}F*\nabla^{(m)}F 
        \\
        &\quad +
        \sum_{i + j = m} \nabla^{(i)}F*\nabla^{(j)}F*\nabla^{(m)}F + \kappa \sum_{i + j + k = m} \nabla^{(i)}F*\nabla^{(j)}F*\nabla^{(l)}F*\nabla^{(k)}F \,,
    \end{align*}
    and claim follows from 
    \[\Delta|\nabla^{(m)}F|^2 = 2 \langle \nabla^{(m)}F, \Delta\nabla^{(m)}F \rangle + 2|\nabla^{(m + 1)}F|^2 \,.
    \]
\end{proof}

\subsection{Riemannian curvature}
We recall that under \eqref{eq: * kappa} the Riemann curvature tensor evolves by
\[\frac{\partial}{\partial t} \Rm = \Delta \Rm + \Rm * \Rm + \kappa \Big(\Rm * F * F + \nabla F * \nabla F + \nabla^2 F * F\Big) \,.
\]
One also has the standard commutation relation
\[[\nabla^{(m)}, \Delta]\Rm = \sum_{i+j=m} \nabla^{(i)}\Rm * \nabla^{(j)}\Rm 
\]
in space, whereas in spacetime we have
\begin{lemma}
    For every $m \in \nN$ we have the commutation identity
    \begin{align*}
        [\partial_t, \nabla^{(m)}]\Rm = \sum_{i + j = m} \nabla^{(i)}\Rm*\nabla^{(j)}\Rm + \kappa \sum_{i + j + k = m} \nabla^{(i)}\Rm*\nabla^{(j)}F*\nabla^{(k)}F
    \end{align*}
\end{lemma}
\begin{proof}
    For $m = 1$ the statement is clear from $\partial_t\Gamma = \nabla\Rm + \kappa \,\nabla F * F$. Assuming the claim is true for $m-1$, that is
    \[[\partial_t, \nabla^{(m-1)}]\Rm = \sum_{i + j = m-1} \nabla^{(i)}\Rm*\nabla^{(j)}\Rm + \kappa \sum_{i + j + k = m-1} \nabla^{(i)}\Rm*\nabla^{(j)}F*\nabla^{(k)}F \,,
    \]
    we compute
    \begin{align*}
        \partial_t\nabla^{(m)}\Rm &= \nabla \partial_t \nabla^{(m-1)}\Rm + \nabla\Rm * \nabla^{(m-1)}\Rm + \kappa \nabla F * F * \nabla^{(m-1)}\Rm
        \\
        &= \nabla^{(m)}\partial_t\Rm + \nabla\Rm * \nabla^{(m-1)}\Rm + \kappa \nabla F * F * \nabla^{(m-1)}\Rm
        \\
        &\quad + \nabla \left(\sum_{i + j = m-1} \nabla^{(i)}\Rm*\nabla^{(j)}\Rm + \kappa \sum_{i + j + k = m-1} \nabla^{(i)}\Rm*\nabla^{(j)}F*\nabla^{(k)}F\right)
        \\
        &= \nabla^{(m)} + \sum_{i + j = m} \nabla^{(i)}\Rm*\nabla^{(j)}\Rm + \kappa \sum_{i + j + k = m} \nabla^{(i)}\Rm*\nabla^{(j)}F*\nabla^{(k)}F \,.
    \end{align*}
\end{proof}
\begin{proposition}
    The $m$-th covariant derivative of the bundle curvature satisfies
    \begin{align*}
        \frac{\partial}{\partial t}\nabla^{(m)}\Rm &= \Delta\nabla^{(m)}\Rm + \sum_{i + j = m} \nabla^{(i)}\Rm*\nabla^{(j)}\Rm
        \\
        &\qquad + \kappa \sum_{i + j + k = m} \nabla^{(i)}\Rm*\nabla^{(j)}F*\nabla^{(k)}F + \kappa \sum_{i+j=m+2} \nabla^{(i)}F *\nabla^{(j)}F \,.
    \end{align*}
\end{proposition}
\begin{proof}
    Using the above,
    \begin{align*}
        \frac{\partial}{\partial t}\nabla^{(m)}\Rm &= \nabla^{(m)}\partial_t\Rm + \sum_{i + j = m} \nabla^{(i)}\Rm*\nabla^{(j)}\Rm + \kappa \sum_{i + j + k = m} \nabla^{(i)}\Rm*\nabla^{(j)}F*\nabla^{(k)}F
        \\
        &= \nabla^{(m)}\Big(\Delta \Rm + \Rm * \Rm + \kappa \bigl(\Rm * F * F + \nabla F * \nabla F + \nabla^2 F * F\bigr)\Big)
        \\
        &\qquad + \sum_{i + j = m} \nabla^{(i)}\Rm*\nabla^{(j)}\Rm + \kappa \sum_{i + j + k = m} \nabla^{(i)}\Rm*\nabla^{(j)}F*\nabla^{(k)}F
        \\
        &= \Delta\nabla^{(m)}\Rm + \sum_{i + j = m} \nabla^{(i)}\Rm*\nabla^{(j)}\Rm
        \\
        &\qquad + \kappa \sum_{i + j + k = m} \nabla^{(i)}\Rm*\nabla^{(j)}F*\nabla^{(k)}F + \kappa \sum_{i+j=m+2} \nabla^{(i)}F *\nabla^{(j)}F \,.
    \end{align*}
\end{proof}
\begin{corollary}
    \label{cor: evol |nablaRm|}
    For every $m \in \nN$ we have the evolution equation
    \begin{align*}
        \frac{\partial}{\partial t} &|\nabla^{(m)}\Rm|^2 \leq \Delta|\nabla^{(m)}\Rm|^2 - 2 |\nabla^{(m + 1)}\Rm|^2 + \sum_{i + j = m} \nabla^{(i)}\Rm*\nabla^{(j)}\Rm*\nabla^{(m)}\Rm 
        \\
        &+ \kappa \sum_{i + j + k = m} \nabla^{(i)}\Rm*\nabla^{(j)}F*\nabla^{(k)}F*\nabla^{(m)}\Rm + \kappa \sum_{i+j=m+2} \nabla^{(i)}F *\nabla^{(j)}F * \nabla^{(m)}\Rm \,.
    \end{align*}
\end{corollary}
    
\end{appendix}

\vspace{1cm}
\printbibliography[]

\end{document}